\documentclass[aop]{imsart}

\RequirePackage{amsthm,amsmath,amsfonts,amssymb}
\RequirePackage[numbers]{natbib}
\RequirePackage{color}

\startlocaldefs
\numberwithin{equation}{section}
\theoremstyle{plain}
\newtheorem{thm}{Theorem}[section]
\newtheorem{prop}[thm]{Proposition}
\newtheorem{lem}[thm]{Lemma}
\newtheorem{cor}[thm]{Corollary}
\theoremstyle{remark}
\newtheorem*{rem}{Remark}
\newtheorem{defn}{Definition}[section]
\newtheorem{cond}{Condition}[section]

\endlocaldefs
\def\op{\mathbf{R}}
\def\act{\mathbb{A}}
\def\I{\mathcal{I}}
\def\e{\mathbb{E}}

\def\g{\zeta}
\def\t{\mathbb{T}^2}
\def\r{\mathbb{R}}
\def\d{\boldsymbol{D}}
\def\p{\mathcal{P}}

\def\lr{\left\langle}
\def\k{\mathcal{K}}

\def\n{\mathcal{N}}
\def\rr{\right\rangle}
\def\x{\mathcal{X}}
\def\curl{{\rm curl}\,}
\def\div{{\rm div}\,}

\begin{document}
\title{Sample-path large deviation principle for a 2-D stochastic  interacting vortex dynamics with singular kernel}
\runtitle{Sample-path LDP for a stochastic
point-vortex model}
\begin{aug}
\author{\fnms{Chenyang} \snm{Chen}\ead[label=e1]{1801110039@pku.edu.cn}},
\author{\fnms{Hao} \snm{Ge}\ead[label=e3,mark]{haoge@pku.edu.cn}}
\address{Beijing International Center for Mathematical Research, Peking University}
\end{aug}

 \let\thefootnote\relax\footnotetext{\printead{e1,e3}}

\begin{abstract}
We consider a stochastic interacting vortex system of $N$ particles, approximating the vorticity formulation of 2-D Navier-Stokes equation on torus. The singular interaction kernel is given by the Biot-Savart law. We only require the initial state to have finite energy, and obtain a sample-path large deviation principle for the empirical measure when the number of vortices goes to infinity. The rate function is characterized by an explicit formula
supporting on sample paths with finite energy and finite integral of $L^2$ norms over time. The proof utilizes a symmetrization technique for the representation of singular kernel, together with a detailed regularity analysis of the sample path with finite rate function. The key step is to prove that the singular term after symmetrization can be bounded by the integral of $L^2$ norms along sample paths.
\end{abstract}



\tableofcontents

\section{Introduction}
The subject of this paper is the sample-path large deviation principle (LDP) for a stochastic interacting vortex dynamics of 2D Navier-Stokes equation starting from initial data with finite energy.

Define $$
r(x,y):=\inf_{k\in \mathbb{Z}^2}|x-y-k|,\quad \forall x,y\in \r^2.
$$
Let $\t=\r^2\backslash \mathbb{Z}^2$ be the two dimensional torus, i.e. the unit cube $[-\frac{1}{2},\frac{1}{2})^2$ with metric $r$, and $\n(x)$ be the Green function on $\t$, i.e.
\begin{equation}\label{equ_green}
-\Delta \n(x)=\delta_0(x)-1,\quad \int_{\t}\n(x)dx=0,
\end{equation}
in which $\delta_0$ means Dirac measure with singularity at $0$. We define a $n-$particle stochastic vortex dynamics by stochastic differential equations on 2D torus for any given $\nu>0$:
\begin{equation}\label{def_sde}
dX_i(t)=\frac{1}{n}\sum_{j\neq i}\k(X_i(t)-X_j(t))dt+\sqrt{2\nu}dB_i(t), \;i=1,2,\cdots,n,\\
\end{equation}
with random initial data $\{X_i(0)\}_{i=1,\cdots,n}$, independent of the family
$(B_i(t))_{i=1,\cdots,n}$ of $i.i.d.$ 2D-Brownian motions, where the Biot-Savart kernel $\k$ is given by $\k=-\nabla ^{\perp}\n=(-\partial_2\n,\partial_1 \n)$. Supposing the initial energy, i.e.
$$
\frac{1}{2n^2}\sum_{i\neq j}\n(X_i(0)-X_j(0)),
$$
is finite, we are interested in the sample-path LDP of the empirical measure process $\rho_n:=\rho_n(t)=\frac{1}{n}\sum_{i=1}^n\delta_{X_i(t)}$ when $n$ goes to $+\infty$.

The model without noise, i.e. $\nu=0$, was originally introduced by Helmholtz \cite{RN208} and then studied by Kirchhoff \cite{RN251}, describing the motion of vortices. Schochet \cite{RN243} proved its convergence to the periodic weak solutions of Euler equation. (\ref{def_sde}) with $\nu=0$ is exactly a Hamiltonian system with the singular potential $\n$ and is closely related to Onsager's theory \cite{RN278} for two dimensional fluids, which was latter developed by Joyce and Montgomery \cite{RN211}. The associated thermodynamic limit of microcanonical/canonical distributions was made rigorous by Caglioti, Lions, Marchioro and Pulvirenti \cite{RN207,RN224} as well as Eyink and Spohn \cite{RN206} in 1990s. The idea behind these theories is exactly LDP for a sequence of distributions. Since then, from the perspective of equilibrium statistical physics, the large deviation principle for more general Gibbs measures with nonsingular and singular Hamiltonians have be extensively studied \cite{RN271,RN270,RN273,RN274,RN269}.

It's well known that the empirical measure process $\{\rho_n(t)\}$ of the corresponding model in $\r^2$ converges to the vorticity form of two dimensional Navier-Stokes equations as $n\rightarrow\infty$.

Consider the incompressible Navier-Stokes equations on $\r^2$:
\begin{equation}\label{equ_ns}
\begin{aligned}
&\partial_t u+(u\cdot \nabla) u+\nabla p=\nu\Delta u, \quad \div u=0,\\
&u(0,x)=u_0,\quad \div u_0=0,
\end{aligned}\end{equation}
Its vorticity $\rho(t):=\curl u(t)=\partial_2 u_1-\partial_1 u_2$ satisfies a continuity equation
\begin{equation}\label{equ_curl1}
\partial_t \rho+\div(\rho u)-\nu\Delta \rho=0,
\end{equation}
with a velocity field driven by the vorticity itself $u(t):=\k*\rho(t)$, in which $\k(x)=\frac{1}{2\pi|x|^2}(-x_2,x_1)$ for the case of $\r^2$. (\ref{equ_curl1}) is called the voricity form of two-dimensional Navier-Stokes equations.

Osada \cite{RN218,RN254} investigated the well posedness of (\ref{def_sde}) in $\r^2$ and obtained a propagation of chaos for the equation (\ref{equ_curl1}) for large $\nu$. For general positive $\nu$,  using the cutoff kernels $\k_n$ converging to the original one with singularity, Marchioro and Pulvirenti \cite{RN258} and M\'el\'eard \cite{RN188} also proved propagation of chaos. Fournier, Hauray and Mischler \cite{RN192} proved a stronger "entropic chaos" result not requiring $\nu$ large and without cutoff when the distributions of initial data are regular enough.

Back to $\t$, the well posedness for this SDEs with the singular kernel $\k$ is firstly studied in \cite{RN244}. In this case, as a classical result, $\n$ are smooth even function on $\t\backslash (0,0)$ and
$$
\n(x)=-\frac{1}{2\pi}\log(|x|)+\n_0(x),
$$
where $\n_0$ is a bounded correction to periodize $\n$ on $\t$. In our paper, we need a refined result that
\begin{equation}\label{equ_N}
\begin{aligned}
&\n(x)=-\frac{\psi(x)}{2\pi}\log |x|+\sigma_1(x),\\
&\k(x)=\frac{\psi(x)}{2\pi|x|^2}(-x_2,x_1)+(\sigma_2(x),\sigma_3(x)),
\end{aligned}
\end{equation}
for certain smooth function $0\leq \psi\leq 1$ satisfying $$
\psi(x)=\left\{
\begin{aligned}
&1,& |x|<\frac{1}{4},\\
&0,&|x|>\frac{1}{3},
\end{aligned}
\right.
$$ and $\sigma_i\in C^\infty(\t)$. The method of \cite{RN192} also works and the propagation of chaos holds for (\ref{equ_curl1}) on $\t$. The slight difference is that $\rho(t):=\curl u(t)+1$ instead of $\curl u(t)$ in the case of $\t$ in which $u(t)$ is the solution of Navier Stokes equation (\ref{equ_ns}) on $\t$.

\subsection{Weakly interacting stochastic particle system with singular kernel}
The 2-D vortex dynamics (\ref{def_sde}) is a weakly interacting stochastic particle model with independent noise. It is called "weak" because when the number of particles tends to infinity, the force between any two particles tends to zero.

For the mean-field limit, the case in the presence of smooth interaction kernels has been well studied since McKean's work \cite{RN255}. But there is only few references in the case of singular interactions for a long time. Recently, based on the estimation for relative entropy, Jabin and Wang \cite{RN242} proved the mean-field limit and obtained the optimal convergent rate of distributions for a general class of singular kernels including the Biot-Savart law. Meanwhile, based on modulated energy, Serfaty \cite{RN229} and Duernckx \cite{RN241} proved the mean-field limit for deterministic systems with more singular kernels. Combining the two methods, Bresch, Jabin and Wang \cite{RN221} derived the mean-field limit for the Patlak-Keller-Segel kernel and gave the corresponding convergent rate.

To quantify the difference between empirical measure processes and mean-field limit more precisely, some considered the fluctuation in path space, usually called the central limit theorem (CLT) for Gaussian fluctuations. Results of this type were firstly obtained by Tanaka
and Hitsuda \cite{RN265} and by Tanaka \cite{RN266} for one dimensional non-singular interacting diffusions. For more general case with regular enough coefficients, it was proved by Fernandez and M\'el\'eard \cite{RN264}. Recently, Wang, Zhao, and Zhu \cite{RN263} proved CLT for a special class of singular interaction including (\ref{def_sde}).

A further topic is the sample-path large deviation principle for empirical measures. For non-singular interaction kernels only in drift coefficients, Dawson and G\"artner \cite{RN257} proved the sample-path LDP for general weakly interacting stochastic particles model using measure transformation from independent diffusions. Budhiraja, Dupuis and Fischer \cite{RN260} further developed this result, allowing the non-singular interaction kernels in both drift and diffusion terms. The method by Feng and Kurtz, based on Hamiltonian-Jocabi theory, also works for smooth kernel (see Section 13.3 of \cite{RN90}).

For singular interaction kernels, Fontbona \cite{RN261} proved the sample-path LDP for diffusing particles with electrostatic repulsion in one dimension. Applying the method in \cite{RN90}, at least formally, Feng and \'Swiech \cite{RN171} gave the rate function of sample-path LDP for empirical measures of (\ref{def_sde}), starting from initial data with finite entropy. However, making this approach rigorous is really subtle, requiring an uniqueness theory for infinite-dimensional Hamiltonian-Jocabi equation, especially when only having initial data with finite energy in hand.

In the present paper, we aim to study the sample-path large deviation principle of (\ref{def_sde}) starting from more rough initial data only with finite energy. We follow the standard approach used in \cite{RN257,RN5,RN96}, i.e., applying an exponential martingale inequality to prove the upper bound and approximating trajectories with finite action (rate function) by a series of mildly perturbed systems for the proof of lower bound. The main difficulty here comes from the term in the rate function containing the singular kernel $\k$.

We use the idea from \cite{RN267,RN143,RN136}, introducing an auxiliary functional as a modification of the rate function to make the approximating strategy for the proof of lower bound work. Moreover, during the proof of large deviation upper bound, this auxiliary functional ensures the continuity of singular terms in the Markov generator with respect to weak topology, which is not a problem for the case of non-singular kernels.

As a price, we need a sharp prior estimation for the auxiliary functional for the upper bound establishment. We finally choose the supremum of energy functional and the integral of $L^2$ norm along time as the auxiliary functional in our paper. Both in the procedure of nice trajectory approximation during the proof of lower bound, and in the prior estimation of auxiliary functional during the proof of upper bound, the singular terms are shown to be bound by the integral of $L^2$ norm over time, which is our key observation and main tools throughout this paper.

\subsection{Notations and definitions}
We write $C^\infty(\Omega)$ and $\mathcal{D}'(\Omega)$ for the collection of infinitely differentiable functions on $\Omega$ and generalized functions on $\Omega$.

We use $\lr\cdot,\cdot\rr$ to represent pairing of linear space and its dual space, such as $L^2$ inner product, continuous function integral with respect to a measure, and so on.

Denote by $\p(\t)$ the space of probability measures on $\t$. For $\gamma,\eta\in \p(\t),$ we define the Wasserstein$-2$ metric
\begin{equation}\label{equ_wassp}
d(\gamma,\eta):=\inf_{\pi\in\Pi(\gamma,\eta)}\int_{\t\times \t}r(x,y)^2\pi(dx,dy),
\end{equation}
where $$\Pi(\gamma,\eta)=\{\pi\in\p(\t\times\t):\pi(dx,\t)=\gamma(dx),\pi(\t,dy)=\eta(dy)\}.$$

Then $(\p(\t),d)$ is a complete separable metric space. See Chapter 7 of \cite{RN217} or Chapter 7 of \cite{RN187} for properties
of this metric.
Given $\mu\in \p(\t)$, for any $m\in \mathcal{D}'(\t)$ we define
\begin{equation}\label{equ_m1mu}
\|m\|_{-1,\mu}^2=\sup_{\phi \in C^\infty(\t)}\left\{2\lr\phi,m\rr-\int_{\t}|\nabla\phi(x)|^2d\mu\right\}.
\end{equation}
Let $AC((s,t);\p(\t))$ be the set of all $\rho\in C([s,t];\p(\t))$ satisfying that there exists $m\in L^1[s,t]$ such that
$$
d(\rho(p),\rho(q))\leq \int_{p}^{q}m(r)dr,\quad \forall s<p\leq q<t.
$$
We refer to smooth mollifier throughout this paper as a non-negative even function that is smooth and its integral over $\t$ is one.

\subsection{Statement of the main result}\label{sec_LDP}
We introduce the state space $$\x_n:=\left\{\rho=\frac{1}{n}\sum_{i=1}^{n}\delta_{x_i}:x_1,\cdots,x_n\in \t,x_i\neq x_j,\forall 1\leq i<j\leq n\right\}.$$
As shown by \cite{RN244}, for the case $X_i(0)=x_i$, the system (\ref{def_sde}) is well-posed for almost every $(x_i)_{i=1,\cdots,n}$, but without an explicit characterization. Actually, we need a well-posed result like Theorem 2.10 in \cite{RN192}, and \cite{RN250} gives a probabilistic proof which can be directly applied to the case of torus and lead to the following lemma.
\begin{lem}\label{lem_wellpose}
Consider $(X_{i}(0))_{i=1,\cdots,n}$ as $\t$-valued random variables, independent of the family
$(B_i(t))_{i=1,\cdots,n}$ of $i.i.d.$ 2D-Brownian motions. Suppose that $P(\rho_n(0)\in\x_n)=1$, then there exists an unique strong solution to (\ref{def_sde02}) and $$P\left(\inf\left\{t:\rho_n(t)\notin \x_n\right\}<\infty\right)=0.$$
\end{lem}
Now we introduce the energy functional for empirical measures
$$
e_0(\gamma):=\frac{1}{2}\int_{(\t)^2\backslash \d}\n(x-y)\gamma(dx)\gamma(dy),\quad \forall \gamma\in \bigcup_{n\geq1}\x_n,
$$
where $\d=\{(x,y)\in (\t)^2:x=y\}$ is the diagonal set. We impose the following superexponential finite energy condition for initial data.

\begin{cond}\label{cond_ldp_1}
$$\lim_{R\to\infty}\limsup_{n\to\infty}\frac{1}{n}\log P\left(e_0(\rho_n(0))>R\right)=-\infty.$$
\end{cond}

The removal of diagonal set is not convenient for estimations, so we also introduce the energy functional containing the diagonal set,
$$
e(\gamma):=\frac{1}{2}\int_{(\t)^2}\n(x-y)\gamma(dx)\gamma(dy),\quad \forall \gamma\in \p(\t).
$$
It's always well defined since $\n$ is bounded below. By Fatou's Lemma, $e$ is lower semi-continuous functional under the weak topology, and we left more properties of $e$ into Appendix \ref{sec_hm1}.

Let $Q_T:C([0,T];\p(\t))\mapsto \r$ be defined by
$$Q_T(\rho)=\sup_{0\leq t\leq T}\left(e(\rho(t))+\frac{\nu}{2}\int_{0}^{t}\|\rho(s)-1\|_2^2ds\right),$$
and for $\rho\in AC((0,T);\p(\t))$ and $Q_T(\rho)<\infty$, we define
$$\act_T(\rho)=\frac{1}{4\nu}\int_{0}^{T}\|\partial_t \rho(t)-\div[\rho(t) (\k*\rho)(t)]+\nu\Delta \rho(t)\|_{-1,\rho(t)}^2dt.$$

To check $\act_T(\rho)$ is well defined, by Theorem 8.3.1 of \cite{RN217}, for $\rho\in AC((0,T);\p(\t))$, we can define $\partial_t \rho\in \mathcal{D}'(\t)$ for almost every $t\in [0,T]$. In addition, $Q_T(\rho)<\infty$ implies that $\rho(t)\in L^2(\t)$ for almost every $t$. As a consequence of (\ref{equ_N}) and Young's inequality, if $\gamma\in L^2(\t)$, then $\k*\gamma\in L^2(\t)$, so $$\lr\phi,\div[\gamma(\k*\gamma)]\rr:=-\int_{\t}\nabla \phi(x)\cdot (\k*\gamma)(x) \gamma(x))dx$$ is well defined for any smooth function $\phi$. Thus $\div[\rho(t)(\k*\rho(t))]\in D'(\t)$ for almost every $t$.

Our main result is
\begin{thm}[Large deviation principle]\label{thm_LDP} Under the condition of Lemma \ref{lem_wellpose},
suppose $\rho_n(0)$, as random variables on $\p(\t)$, satisfies a large deviation principle with the rate function $\I_0$ and Condition \ref{cond_ldp_1} holds. Then for each $T>0$, the stochastic empirical process $\{\rho_n(t):0\leq t\leq T\}_{n=1,2,\cdots}$ satisfies a sample-path large deviation principle in the space $C([0,T];\p(\t))$ with a good rate function (usually called action)
\begin{equation}\label{equ_ratefunction}
\begin{aligned}
\I_T(\rho)=\left\{
\begin{aligned}
&\I_0(\rho(0))+\act_T(\rho),&\text{if} ~\rho\in AC((0,T);\p(\t))~ \text{and}~ Q_T(\rho)<\infty,\\
&\infty, &otherwise,
\end{aligned}
\right.
\end{aligned}
\end{equation}
that is equivalent to the following two conditions:

1. (Large deviation upper bound) For any closed set $A$ in $C([0,T];\p(\t))$,
$$
\begin{aligned}
\limsup_{n\to \infty}\frac{1}{n}\log P(\rho_n\in A)\leq -\inf_{\rho\in A}\I_{T}(\rho);
\end{aligned}
$$

2. (Large deviation lower bound) For any open set $G$ in $C([0,T];\p(\t))$,
$$
\begin{aligned}
\liminf_{n\to \infty}\frac{1}{n}\log P(\rho_n\in G)\geq -\inf_{\rho\in G}\I_{T}(\rho).
\end{aligned}
$$

\end{thm}

\subsection{Outline of proof}

As in the seminar works \cite{RN257,RN5,RN96}, the sample-path large deviation upper bound can be obtained through an exponential martingale estimation and exponential tightness. Additional treatment is required in the presence of singular terms, i.e. a prior energy estimation to make sure the generator is continuous under weak typology.

The basic strategy in \cite{RN5,RN96} to prove lower bound is to compute the relative entropy between a model with regular perturbation and the original one, which actually provides a lower bound on the probability of the neighborhood of the mean-field limit path of the perturbed process. However, this strategy for lower bound estimation would fail in general when the rate function is not convex, let alone in the presence of singular interaction kernels, unless there is an additional modification for the rate function, as have been done in \cite{RN267,RN143,RN136}. As we have stated in the main theorem, we choose $Q_T<\infty$ as the modification, which implies finite energy and finite integral of $L^2$ norm of the distributions over time.

Besides the modification by $Q_T<\infty$, the proof for the lower bound still needs several new inequalities, which can bound the singular terms such as $\div [\rho (\k*\rho)]$ by the integral of $L^2$ norms along sample path. These inequalities are originated from an energy dissipation structure similar to the classical result for 2D Navier-Stokes equation (e.g. see (3.20) in \cite{RN275}), in which the integral of $L^2$ norms along sample path naturally emerges. We can not expect the $L^2$ norms to be bounded from above since the initial $L^2$ norm  can be infinite, but the integral of $L^2$ norm along sample path with finite action is proven to be finite, utilizing such an energy dissipation structure. As one may expect, the finite energy condition is required for the energy dissipation structure, and as far as we know, this is the weakest condition that makes the mean-field limit holds for (\ref{def_sde}) with general $\nu>0$.

As stated in \cite{RN136}, the rate function with modification makes it harder to prove the upper bound, requiring a prior estimation of $Q_T$. However, $Q_T$ always takes $\infty$ on the empirical measure $\rho_n\in\x_n$, even if their limit as $n\to\infty$ has finite $Q_T$. Thus we make use of the convolution method with a sequence of smooth mollifiers. By taking specific mollifiers $\g_n$ converging to Dirac measure with a suitable slow rate, we yield an energy dissipation structure for $\g_n*\rho_n$ assuming $e(\g_n*\rho_n(0))$ is finite, which gives the control for $Q_T(\g_n*\rho_n)$. The key ingredient in the proof is a sharp estimation for the singular term in order to show its vanishing in the energy dissipation structure. Moreover, the energy dissipation structure, as well as related inequalities, can also help to provide a more explicit formula of the rate function than the variational one directly derived from the exponential martingale inequality.

Furthermore, in order to characterize the limit of infintesimal operators of the SDEs with singular terms, we utilize a symmetrization technique to first represent the limit operator defined on arbitrary probability measure, then perform the regularity analysis, and finally come back to the original representation without symmetrization.

In more detail, take $\g:\r\mapsto \r^+:$$$\g(x)=\left\{
\begin{aligned}
&Ce^{-\frac{1}{1-4|x|^2}},& |x|<\frac{1}{2},\\
&0,& |x|\geq \frac{1}{2},
\end{aligned}\right.
$$ where the constant is to make $\int_{0}^{\infty}2\pi r \g(r)dr=1$. $\g(x)$ is a non-negative smooth function on $\r$ with a compact support. Then let $\g_n:\r^2\mapsto \r$ be the specific smooth mollifiers generated by $\g$ and defined by
$$
\g_n(x)=m_n^2\g(m_n|x|),
$$
where $m_n\uparrow \infty,nm_n^{-2}\to \infty$. Note that the slow convergent rate of $m_n$ to infinity is crucial for the prior estimation that we will show later.

In Section \ref{sec_proof} we will illustrate the proof for the main theorem under the following superexponential finite energy condition for $\g_n*\rho_n$ that we will show later.
\begin{cond}\label{cond_ldp_2}
$$\lim_{R\to\infty}\limsup_{n\to\infty}\frac{1}{n}\log P\left(e(\g_n*\rho_n(0))>R\right)=-\infty.$$
\end{cond}
\begin{thm}[Large deviation principle \uppercase\expandafter{\romannumeral2}]\label{thm_LDPplus}  Under the condition of Lemma \ref{lem_wellpose}, suppose $\rho_n(0)$ satisfies a large deviation principle on $\p(\t)$ with the rate function $\I_0$ and Condition \ref{cond_ldp_2} holds. Then for each $T>0$, the stochastic empirical process $\{\rho_n(t):0\leq t\leq T\}_{n=1,2,\cdots}$ satisfies a sample-path large deviation principle in the space $C([0,T];\p(\t))$ with a good rate function given by (\ref{equ_ratefunction}).
\end{thm}

To conclude the proof, we will also show that finite $e_0(\rho(0))$ is stronger than finite $e(\g_n*\rho(0))$, i.e. the Condition \ref{cond_ldp_1} implies the Condition \ref{cond_ldp_2}, in Section \ref{sec_ener}.

In Section \ref{sec_ener}, we also prove several crucial inequalities for bounding the singular term $\div \rho(\k*\rho)$ under convolution and give the prior estimation of $Q_T$ for the upper bound proof. Section \ref{sec_regular} is about regularities of trajectories with finite rate function, which helps us to express the rate function in a more explicit way. In Section \ref{sec_lln}, we prove the law of large number for the model with certain regular perturbation and perform the "nice" trajectory approximation strategy for the lower bound proof.

\section{Proof of the main theorem}\label{sec_proof}
In this section we present the main proof procedure of Theorem \ref{thm_LDPplus}, and we place all the proofs of lemmas into later sections.
\subsection{Exponential martingale problem}
Our proof of large deviation upper bound is based on the analyses of the exponential martingale problem.

Recall that the dynamics of $X(t)$ are determined by the stochastic differential equations
\begin{equation}\label{def_sde02}
dX_i(t)=\frac{1}{n}\sum_{j\neq i}\k(X_i-X_j)dt+\sqrt{2\nu}dB_i(t), \;i=1,2,\cdots,n.\\
\end{equation}

By Ito's formula, for each smooth function $\phi$ on $\t$, we have
$$\begin{aligned}
&d\left[\lr\phi,\rho_n(t)\rr\right]=\frac{1}{n}\sum_{i=1}^nd\phi(X_i(t))=\frac{1}{n}\sum_{i=1}^n\nabla \phi(X_i(t))dX_i(t)
+\frac{\nu}{n}\sum_{i=1}^n\Delta \phi(X_i(t))dt\\
&=\frac{1}{n^2}\sum_{i\neq j}\nabla \phi(X_i(t))\cdot\k(X_i(t)-X_j(t))dt+\frac{\sqrt{2\nu}}{n}\sum_{i=1}^{n}\nabla \phi(X_i(t))dB_i(t)\\
&+\frac{\nu}{n}\sum_{i=1}^n \Delta \phi(X_i(t))dt
\end{aligned}
$$

To write the generator of $\rho_n$, we need to write the singular term $\frac{1}{n^2}\sum_{i\neq j}\nabla \phi(X_i(t))\cdot\k(X_i(t)-X_j(t))$ as a $\rho_n(t)$ dependent term. Here we use the observation from Delort \cite{RN262}, proposing an alternative representation for singular term which can be easily extended to general measure.

Recalling $\d$ is the diagonal set, noting $\k(x)=-\k(-x),$
\begin{equation}\label{equ_Delort}
\begin{aligned}
&\frac{1}{n^2}\sum_{i\neq j}\nabla \phi(X_i(t))\cdot\k(X_i(t)-X_j(t))\\
&=\frac{1}{2n^2}\sum_{i\neq j}\left[\nabla \phi(X_i(t))-\nabla \phi(X_j(t))\right]\cdot\k(X_i(t)-X_j(t))dt\\
&=\frac{1}{2}\int_{(\t)^2\backslash \d}(\nabla \phi(x)-\nabla \phi(y))\cdot \k(x-y) \rho_n(t,dx)\rho_n(t,dy)\\
&=\frac{1}{2}\int_{(\t)^2\backslash \d}\frac{\nabla\phi(x)-\nabla\phi(y)}{r(x,y)}\cdot r(x,y)\k(x-y) \rho_n(t,dx)\rho_n(t,dy)
\end{aligned}
\end{equation}
is only dependent on $\nabla\phi$ and $\rho_n(t)$ and linear of $\nabla \phi.$  In view of (\ref{equ_N}), $w(x,y)=r(x,y)\k(x-y)$ is a bounded function so that for each $\varphi\in C^\infty(\t;\r^2)$,
\begin{equation}
|(\varphi(x)-\varphi(y))\cdot \k(x-y)|\leq C_{\k}\|\nabla \varphi\|_{\infty}
\end{equation}
is bounded. Hence the final expression of (\ref{equ_Delort}) is always well defined and can be extended to any $\gamma\in \p(\t)$. Now we introduce the symmetrization operator $\op:\p(\t)\mapsto \mathcal{D}'(\t)$ satisfying
\begin{equation}\label{equ_rucv1}
\lr\varphi,\op(\gamma)\rr:=\frac{1}{2}\int_{(\t)^2\backslash \d}(\varphi(x)-\varphi(y))\cdot \k(x-y) \gamma(dx)\gamma(dy),\,\forall \varphi\in C^1(\t;\r^2).
\end{equation}
As stated above, (\ref{equ_rucv1}) is always well-defined and
\begin{equation}\label{equ_opbound}
\left|\lr\varphi,\op(\gamma)\rr\right|\leq \frac{1}{2}C_{\k}\|\nabla \varphi\|_{\infty}.
\end{equation}
In fact $\op(\gamma)$ is an extension for $\gamma (\k*\gamma).$ If
$\gamma\in L^2(\t),$ as stated before Theorem \ref{thm_LDP}, $\k*\gamma\in L^2(\t)$. Then for each $\varphi\in C^1(\t;\r^2)$, $$
\lr\varphi,\gamma (\k*\gamma)\rr=\int_{\t}\varphi(x)\cdot\k(x-y)\gamma(dx)\gamma(dy)=-\int_{\t}\varphi(y)\cdot\k(x-y)\gamma(dx)\gamma(dy).
$$
Hence
\begin{equation}\label{equ_rucv}
\op(\gamma)=\gamma (\k*\gamma),\quad \forall \gamma\in L^2(\t).
\end{equation}

With this notation, after further calculations, we can write the generator for $\rho_n$ as
\begin{equation}\label{equ_An1}
A_n f(\gamma):=\sum_{i=1}^{k}\lr \nabla \varphi_i,\op(\gamma)\rr\partial_i \psi+\nu\sum_{i=1}^{k}\lr\Delta \varphi_i,\gamma\rr\partial_i \psi
+\frac{\nu}{n}\sum_{i,j=1}^{k}\lr \nabla \varphi_i \cdot \nabla \varphi_j,\gamma\rr \partial_{ij}\psi,
\end{equation}
for each $f\in D_0:=\{f:\p(\t)\mapsto \r, f(\gamma)=\psi(\lr\varphi_1,\gamma\rr,\cdots,\lr\varphi_k,\gamma\rr),\psi\in C^2(\r^{k}),\varphi_i\in C^\infty(\t)\}$ such that $$
f(\rho_n(t))-\int_0^t A_nf(\rho_n(s))ds=\text{martingale.}
$$
Then we define
\begin{equation}\label{equ_Hn1}
\begin{aligned}
&H_n f(\gamma):=\frac{1}{n} e^{-nf}A_ne^{nf}(\gamma)=\lr\nu\Delta\gamma-\div \op(\gamma),\frac{\delta f}{\delta \gamma}\rr
+\nu\int_{\t}\left|\nabla \frac{\delta f}{\delta \gamma}\right|^2d\gamma\\
&+\frac{\nu}{n}\sum_{i,j=1}^{k}\lr \nabla \varphi_i \cdot \nabla \varphi_j,\gamma\rr \partial_{ij}\psi,
\end{aligned}
\end{equation}
in which $\Delta \gamma$ as well as $\div \op(\gamma)$ is defined in distribution sense and the variational derivative $$\frac{\delta f}{\delta \gamma}(x):=\sum_{i=1}^k\partial_i\psi(\lr\varphi_1,\gamma\rr,\cdots,\lr\varphi_k,\gamma\rr)\varphi_i(x)$$ is a smooth function on $\t$.

At least formally by Ito's formula, for $(f_t)_{t\geq 0}$ such that $f_t\in D_0$, the exponential form
$$
\exp\left\{n\left[f_t(\rho_n(t))-f_t(\rho_n(0))-\int_0^t \partial_s f_s(\rho_n(s))ds-\int_0^t H_nf_s(\rho_n(s))ds\right]\right\}
$$
is always considered to be a martingale. In fact, the proof of \cite{RN5,RN96} only use the case when $\phi\in C^\infty([0,T]\times \t)$ and $f_t(\rho)=\lr \phi(t),\rho\rr.$ To adapt their proof to the more general initial data, we shall use a bigger domain $D_0$, which is sufficient to determine the rate function (see Section 3.2 of \cite{RN90}).

\subsection{Proof of upper bound}
First, as a standard step to prove large deviation principle, we prove the exponential tightness of $\rho_n$, with which we only need to prove the upper bound for compact sets.
\begin{prop}[Exponentially tightness]\label{prop_expc}
Under the condition of Theorem \ref{thm_LDPplus}, for each $a,T>0,$ there exists a compact set $\hat{K}_{a,T}\subset C([0,T];\p(\t))$ such that
\begin{equation}\label{equ_expc}
\limsup_{n\to\infty}\frac{1}{n}\log P(\rho_n \notin \hat{K}_{a,T})\leq -a.
\end{equation}
\end{prop}
\begin{proof}
Note $D_0$ is closed under addition and separates points. In addition, recalling (\ref{equ_Hn1}) and (\ref{equ_opbound}), for
$f(\gamma)=\psi(\lr\varphi_1,\gamma\rr,\cdots,\lr\varphi_k,\gamma\rr),\psi\in C^2(\r^{k}),\varphi_i\in C^\infty(\t)$
\begin{equation}\label{equ_Hnbound}
\begin{aligned}
|H_nf(\gamma)|\leq \nu \left\|\Delta \frac{\delta f}{\delta \gamma}\right\|_{\infty}+\nu
\left\|\nabla \frac{\delta f}{\delta \gamma}\right\|^2_{\infty}+C\left\|\nabla^2 \frac{\delta f}{\delta \gamma}\right\|_{\infty}\\+\frac{\nu}{n}\sum_{i,j=1}^{k}\| \nabla \varphi_i\|_{\infty}  \|\nabla \varphi_j\|_{\infty} \sup_{|x_i|\leq \|\varphi_i\|_{\infty}}|\partial_{ij}\psi(x_1,\cdots,x_k)|
\end{aligned}
\end{equation}
is bounded. Since $\p(\t)$ is a compact space, we conclude the proof by Corollary 4.17 of \cite{RN90}.
\end{proof}

Formally, there are three steps to prove the large deviation upper bound for compact sets.
First, by Ito's formula, for each $f\in D_0$, $\phi\in C^\infty([0,T]\times \t)$ and $t_0\in(0,T)$,
$$
\begin{aligned}
&M_n(t)
:=\exp\bigg\{n\bigg[f(\rho_n(t_{0}\wedge t))-\int_{0}^{t_0\wedge t} H_nf(\rho_n(s))ds
-\int_{t_0\wedge t}^{t}\lr\partial_t \phi(s),\rho_n(s)\rr ds\\
&+\lr\phi(t),\rho(t)\rr-\lr \phi(t_0\wedge t),\rho(t_0\wedge t)\rr
-\int_{t_0\wedge t}^{t}H_n (\lr\phi(s),\cdot\rr)(\rho(s))ds
\bigg]\bigg\}
\\&=\exp\bigg\{n\bigg[f(\rho_n(t_{0}\wedge t))-\int_{0}^{t_0\wedge t} H_nf(\rho_n(s))ds
-\int_{t_0\wedge t}^{t}\lr\partial_t \phi(s),\rho_n(s)\rr ds\\
&+\lr\phi(t),\rho(t)\rr-\lr \phi(t_0\wedge t),\rho(t_0\wedge t)\rr
-\nu\int_{t_0\wedge t}^{t}\lr\Delta \phi(s),\rho_n(s)\rr ds\\
&-\int_{t_0\wedge t}^{t}\lr\nabla \phi(s),\op(\rho_n(s))\rr ds
-\nu\int_{t_0\wedge t}^{t}\int_{\t}|\nabla \phi(s)|^2d\rho_n(s)ds
\bigg]\bigg\}
\end{aligned}
$$
is a non-negative martingale, hence for each measurable set $B\subset C([0,T];\p(\t))$,
\begin{equation}\label{equ_upbdout1}
\begin{aligned}
&\frac{1}{{n}}\log P(\rho_{n}\in B)=\frac{1}{n}\log \e \left[M_{n}(T)(M_{n}(T))^{-1}\chi_{B}(\rho_{n})\right] \\
&\leq \frac{1}{n}\log \e M_n(0)-\inf_{\rho \in B}\frac{1}{n}\log M_n(T)\\
&=\frac{1}{n}\log \e e^{nf(\rho_n(0))}-\inf_{\rho \in B}\left(\mathcal{A}_{n,T}(\rho,f,t_0,\phi)+f(\rho(0))\right),
\end{aligned}
\end{equation}
where $\chi$ is the indicator function and
\begin{equation}\label{equ_oatred}
\begin{aligned}
&\mathcal{A}_{n,T}(\rho,f,t_0,\phi):=f(\rho(t_0))-f(\rho(0))-\int_{0}^{t_0} H_nf(\rho(t))dt
-\int_{t_0}^{T}\lr\partial_t \phi(t),\rho(t)\rr dt\\
&+\lr\phi(T),\rho(T)\rr-\lr \phi(t_0),\rho(t_0)\rr
-\nu\int_{t_0}^{T}\lr\Delta \phi(t),\rho(t)\rr dt\\
&-\int_{t_0}^{T}\lr\nabla \phi(t),\op(\rho(t))\rr dt
-\nu\int_{t_0}^{T}\int_{\t}|\nabla \phi(t)|^2d\rho(t)dt.
\end{aligned}
\end{equation}

Second, obtain the limit of the right-hand-side of (\ref{equ_upbdout1}) when $n$ goes to infinity (write it as $F_{t_0,T,f,\phi}$). Taking the supremum over $\{(f,t_0,\phi):f\in D_0,0<t_0<T,\phi\in C^\infty([0,T]\times \t)\}$, and then we have
$$
\begin{aligned}
\limsup_{n\to\infty}\frac{1}{n}\log P(\rho_n\in B)\leq -\sup_{f,t_0,\phi}\inf_{\rho\in B}F_{t_0,T,f,\phi}(\rho).
\end{aligned}
$$

Finally, use Lemma $A2.3.3$ in \cite{RN79} to exchange the supremum and infimum to obtain the upper bound for any compact set $K$, i.e.
$$
\begin{aligned}
\limsup_{n\to\infty}\frac{1}{n}\log P(\rho_n\in K)\leq -\inf_{\rho\in K}\sup_{f,t_0,\phi}F_{t_0,T,f,\phi}(\rho).
\end{aligned}
$$
However, the exchange of the supremum and infimum requires that $F$ is lower semi-continuous under the topology of $C([0,T];\p(\t))$. Hence we do need the singular term  $\op$ is continuous under weak topology. However, that is not true. Fortunately, $\op$ is continuous under weak topology with finite energy, stated in the following lemma.

\begin{lem}\label{lem_rucv}
If $\gamma_n,\gamma \in \p(\t), \lim_{n\to\infty}d(\gamma_n,\gamma)=0$, and $\sup_n e(\g_n*\gamma_n)<\infty$ (or $\sup_n e(\gamma_n)<\infty$), then for each $\varphi_n,\varphi\in C^1(\t,\r^2)$ such that $\sup_{n} \|\nabla \varphi_n\|_{\infty}<\infty$ and $\lim_{n\to\infty}\|\varphi_n-\varphi\|_{\infty}=0$, we have
$$\lim_{n\to \infty}\lr\varphi_n,\op(\gamma_n)\rr=\lr \varphi,\op(\gamma)\rr,\,\forall \varphi\in C^1(\t,\r^2).$$
\end{lem}

The proof of this lemma will be given in Sec.5, since it will also be used for the proof of large deviation lower bound.

Thus, we need a prior energy estimation to bound the energy of the trajectories.
\begin{lem}[Prior energy estimation]\label{lem_Q1}
Under the condition of Theorem \ref{thm_LDPplus},
\begin{equation}\label{equ_lem_Q1}
\begin{aligned}
\lim_{R\to\infty}\limsup_{n\to\infty}&\frac{1}{n}\log P (Q_T(\g_n*\rho_n)>R)=-\infty.
\end{aligned}
\end{equation}
\end{lem}

\begin{prop}[Upper bound for compact sets]\label{prop_upperbound1}
Under the condition of Theorem \ref{thm_LDPplus}, for any compact set $K$ in $C([0,T];\p(\t))$,
$$
\begin{aligned}
\limsup_{n\to \infty}\frac{1}{n}\log P(\rho_n\in K)\leq -\inf_{\rho\in K}\left(\mathcal{I}_0(\rho(0))+ \overline{\act}_{T}(\rho)+\infty\cdot\chi_{Q_T(\rho)=\infty}\right),
\end{aligned}
$$
where
\begin{equation}\label{equ_oat}
\begin{aligned}
&\overline{\act}_{T}(\rho):=\sup_{\phi\in C^\infty([0,T]\times \t)}\bigg(
\lr\phi(T),\rho(T)\rr-\lr\phi(0),\rho(0)\rr
-\int_{0}^{T}\lr\partial_t \phi(s),\rho(s)\rr ds\\
&-\nu\int_{0}^{T}\lr\Delta \phi(s),\rho(s)\rr ds
-\int_{0}^{T}\lr\nabla \phi(s),\op(\rho(s))\rr ds
-\nu\int_{0}^{T}\int_{\t}|\nabla \phi(s)|^2d\rho(s)ds\bigg).
\end{aligned}
\end{equation}
\end{prop}
\begin{proof}
Take $\hat{K}_{a,T}$ in Proposition \ref{prop_expc}. For any measurable set $B\subset C([0,T];\p(\t))$, we define $$B_{n,R,a}=\{\rho:Q_T(\g_n*\rho)\leq R\}\cap B\cap \hat{K}_{a,T}.$$
By Lemma \ref{lem_Q1}, as $R\to \infty,$
\begin{equation}\label{equ_Bnr}
a_{R}':=-\limsup_{n\to\infty}\frac{1}{n}\log P(Q_T(\g_n*\rho_n)>R)\to \infty.
\end{equation}
In view of (\ref{equ_upbdout1}), $$
\frac{1}{n}\log(\rho_n\in B_{n,R,a})\leq \frac{1}{n}\log \e e^{nf(\rho_n(0))}-\inf_{\rho \in B_{n,R,a}}\left(\mathcal{A}_{n,T}(\rho,f,t_0,\phi)+f(\rho(0))\right).
$$

On the one hand, for $f\in D_0$, define $Hf(\gamma)=\lr\nu\Delta\gamma-\div \op(\gamma),\frac{\delta f}{\delta \gamma}\rr+\nu\int_{\t}\left|\nabla\frac{\delta f}{\delta \gamma}\right|^2d\gamma$ and it's easy to obtain that $|H_nf(\gamma)-Hf(\gamma)|\leq \frac{C_f}{n}$ for certain constant $C_f$. On the other hand, by Varadhan's lemma,
$$
\Lambda_0(f):=\lim_{n\to \infty}\frac{1}{n}\log \e e^{f(\rho_n(0))}
$$
exists. Hence,
\begin{equation}\label{equ_m1nlogP}
\begin{aligned}
&\limsup_{n\to \infty}\frac{1}{{n}}\log P(\rho_{n}\in B_{n,R,a})
\leq-\liminf_{n\to\infty}\inf_{\rho \in B_{n,R,a}}\bigg\{\mathcal{A}_{T}(\rho,f,t_0,\phi)+f(\rho(0))-\Lambda_0(f)\bigg\},
\end{aligned}
\end{equation}
where
$$
\begin{aligned}
&\mathcal{A}_{T}(\rho,f,t_0,\phi):=f(\rho(t_{0}))-f(\rho(0))-\int_{0}^{t_0} Hf(\rho(s))ds
+\lr\phi(T),\rho(T)\rr\\
&-\lr\phi(t_0),\rho(t_0)\rr
-\int_{t_0}^{T}\lr\partial_t \phi(s),\rho(s)\rr ds
-\nu\int_{t_0}^{T}\lr\Delta \phi(s),\rho(s)\rr ds\\
&-\int_{t_0}^{T}\lr\nabla \phi(s),\op(\rho(s))\rr ds
-\nu\int_{t_0}^{T}\int_{\t}|\nabla \phi(s)|^2d\rho(s)ds.
\end{aligned}
$$
Now we want to prove that
\begin{equation}\label{equ_m1nlogP1}
\begin{aligned}
\limsup_{n\to \infty}\frac{1}{{n}}&\log P(\rho_{n}\in B_{n,R,a})
\\&\leq-\inf_{\rho \in B}\bigg\{\mathcal{A}_{T}(\rho,f,t_0,\phi)+f(\rho(0))-\Lambda_0(f)+\infty\cdot\chi_{Q_T(\rho)> R}\bigg\}.
\end{aligned}
\end{equation}
It is trivial if there exists $n_0$ such that $B_{n,R,a}=\emptyset$ for $n\geq n_0$. Otherwise, pick $\tilde{\rho}_{n}\in B_{n,R,a}$ approximating the infimum in (\ref{equ_m1nlogP}) for each $n$ satisfying $B_{n,R,a}\neq \emptyset$, i.e.
$$
\mathcal{A}_{T}(\tilde{\rho}_{n},f,t_0,\phi)+f(\tilde{\rho}_{n}(0))\leq
\inf_{\rho \in B_{n,R,a}}\bigg\{\mathcal{A}_{T}(\rho,f,t_0,\phi)+f(\rho(0))\bigg\}+\frac{1}{n}.
$$
Since $B_{n,R,a}\subset \hat{K}_{a,T}\cap \overline{B}$ and $Q_T$ is lower semi-continuous, we can find $\tilde{\rho}\in \overline{B}\cap \hat{K}_{a,T}$, a limiting point of $\tilde{\rho}_{n}$, with $Q_T(\tilde{\rho})\leq R$. Then by Lemma \ref{lem_rucv} and dominated convergence theorem,
$$
\begin{aligned}
&\mathcal{A}_{T}(\tilde{\rho},f,t_0,\phi)+f(\tilde{\rho}(0))=
\liminf_{n\to\infty}\left(\mathcal{A}_{T}(\tilde{\rho}_{n},f,t_0,\phi)+f(\tilde{\rho}_{n}(0))\right)\\
&\leq  \liminf_{n\to\infty}\inf_{\rho \in B_{n,R,a}}\bigg\{\mathcal{A}_{T}(\rho,f,t_0,\phi)+f(\rho(0))\bigg\},
\end{aligned}
$$
which combining with (\ref{equ_m1nlogP}) implies
\begin{equation}
\begin{aligned}
\limsup_{n\to \infty}\frac{1}{{n}}&\log P(\rho_{n}\in B_{n,R,a})
\\
&\leq-\inf_{\rho \in \overline{B}}\bigg\{\mathcal{A}_{T}(\rho,f,t_0,\phi)+f(\rho(0))-\Lambda_0(f)+\infty\cdot\chi_{Q_T(\rho)> R}\bigg\}.
\end{aligned}
\end{equation}
Use the fact $\mathcal{A}_{T}(\rho,f,t_0,\phi)$ is continuous when $Q_T(\rho)\leq R$, which can be obtained by Lemma \ref{lem_rucv} and dominated convergence theorem, then we arrive at (\ref{equ_m1nlogP1})

So far, by (\ref{equ_expc}), (\ref{equ_Bnr}) and (\ref{equ_m1nlogP1}), for each $a,R>0$,
$$
\begin{aligned}
&\limsup_{n\to \infty}\frac{1}{{n}}\log P(\rho_{n}\in B)
\leq \max\bigg\{\limsup_{n\to \infty}\frac{1}{{n}}\log P(\rho_{n}\in B_{n,R,a}),\\
&\limsup_{n\to \infty}\frac{1}{{n}}\log P(\rho_{n}\in \hat{K}_{a,T}^c),\limsup_{n\to \infty}\frac{1}{{n}}\log P(Q_T(\g_n*\rho_n)>R)
\bigg\}\\
&\leq -\inf_{\rho \in B}\min\bigg\{\mathcal{A}_{T}(\rho,f,t_0,\phi)+f(\rho(0))-\Lambda_0(f)+\infty\cdot\chi_{Q_T(\rho)>R},a,a'_R\bigg\}.
\end{aligned}
$$
Taking $a\to \infty$, we obtain \begin{equation}\label{equ_m1nlogP2}
\begin{aligned}
&\limsup_{n\to \infty}\frac{1}{{n}}\log P(\rho_{n}\in B) \\
&\leq -\inf_{\rho \in B}\min\bigg\{\mathcal{A}_{T}(\rho,f,t_0,\phi)+f(\rho(0))-\Lambda_0(f)+\infty\cdot\chi_{Q_T(\rho)>R},a'_R\bigg\},
\end{aligned}
\end{equation}
where the expression inside infimum is lower semi-continuous with respect to $\rho$. Taking the infimum with respect to $f,t_0$ and $\phi$ in the RHS of (\ref{equ_m1nlogP2}) over $\{(f,t_0,\phi):f\in D_0,0<t_0<T,\phi\in C^\infty([0,T]\times \t)\}$, by Lemma A2.3.3 in \cite{RN79}, we can exchange the order of infimum and supremum for compact sets. As consequence, for any compact set $K$, by (\ref{equ_Bnr}), we have
\begin{equation}\label{equ_atatat1}
\begin{aligned}
&\limsup_{n\to \infty}\frac{1}{n}\log P(\rho_n\in K)\leq-\lim_{R\to \infty}
\inf_{\rho \in K}\sup_{f,t_0,\phi}\min\big\{\mathcal{A}_{T}(\rho,f,t_0,\phi)+f(\rho(0))-\Lambda_0(f)\\&+\infty\cdot\chi_{Q_T(\rho)>R},a'_R\big\}\\
&\leq-\lim_{R\to \infty}
\inf_{\rho \in K}\sup_{f,t_0,\phi}\min\big\{\mathcal{A}_{T}(\rho,f,t_0,\phi)+f(\rho(0))-\Lambda_0(f)+\infty\cdot\chi_{Q_T(\rho)=\infty},a'_R\big\}\\
&= -\inf_{\rho \in K}\sup_{f,t_0,\phi}\big\{\mathcal{A}_{T}(\rho,f,t_0,\phi)+f(\rho(0))-\Lambda_0(f)+\infty\cdot\chi_{Q_T(\rho)=\infty}\big\}.
\end{aligned}
\end{equation}

Now we only remain to show for each $\rho\in C([0,T];\p(\t)),$
\begin{equation}\label{equ_atatat}
\begin{aligned}
\sup_{f,t_0,\phi}\left(\mathcal{A}_{T}(\rho,f,t_0,\phi)+f(\rho(0))-\Lambda_0(f)\right)
\geq \I_0(\rho(0))+\overline{\act}_T(\rho).
\end{aligned}
\end{equation}
By Proposition 3.17 of \cite{RN90}, we can find $f_n\in D_0$ such that
\begin{equation}\label{equ_atata1}
\lim_{n\to\infty} \left(f_n(\rho(0))-\Lambda_0(f_n)\right)= \I_0(\rho(0)).
\end{equation}
By (\ref{equ_opbound}), we know $|Hf(\rho(t))|$ is bounded. Thus, for $\rho\in C([0,T];\p(\t)),$
\begin{equation}\label{equ_atata2}
\sup_{\phi\in C^\infty([0,T]\times \t)}\lim_{n\to\infty} \lim_{t\to 0+}\mathcal{A}_{T}(\rho,f_n,t,\phi)=\overline{\act}_T(\rho).
\end{equation}
Combining (\ref{equ_atata1}) with (\ref{equ_atata2}), we have
$$
\begin{aligned}
&\text{LHS of (\ref{equ_atatat})}\geq \sup_{\phi\in C^\infty([0,T]\times \t)} \lim_{n\to\infty}\lim_{t\to 0+}\left(\mathcal{A}_{T}(\rho,f_n,t,\phi)+f_n(\rho(0))-\Lambda_0(f_n)\right)\\
&\geq\overline{\act}_T(\rho)+\I_0(\rho(0)),
\end{aligned}
$$
which together with (\ref{equ_atatat1}) completes the proof.
\end{proof}

With the prior estimation $Q_T(\rho)<\infty$, we can obtain the upper bound is, as shown in the lemma below, actually equal to $\I_T(\rho)$ in (\ref{equ_ratefunction}).
\begin{lem}[Variational representation of the rate function]\label{lem_vrotrf}
If $\rho\in C([0,T];\p(\t))$ such that  $Q_T(\rho)<\infty$ and $\overline{\act}_T(\rho)<\infty$, then $\rho\in AC((0,T);\p(\t))$. Conversely, if $\rho\in AC((0,T);\p(\t))$ with $Q_T(\rho)<\infty$, then $\act_T(\rho)=\overline{\act}_T(\rho).$=
\end{lem}

Finally, we obtain the large deviation upper bound.
\begin{prop}[Upper bound]\label{prop_upper}
Under the condition of Theorem \ref{thm_LDPplus}, for each closed set $A$ in $C([0,T];\p(\t)),$
\begin{equation}\label{equ_upbd200}
\limsup_{n\to\infty}\frac{1}{n}\log P(\rho_n\in A)\leq -\inf_{\rho\in A}\I_T(\rho).
\end{equation}
\end{prop}
\begin{proof}
In view of Proposition \ref{prop_upperbound1}, we have proved the upper bound for  compact sets. By Lemma \ref{lem_vrotrf}, we verified this upper bound is equal to $-\inf_{\rho\in B}\I_T(\rho)$. Finally we conclude the proof by Proposition \ref{prop_expc} and Lemma 1.2.18 in \cite{RN226},
\end{proof}
\subsection{Proof of lower bound}
We firstly study the law of large numbers of a regularly perturbed model, which can be obtained by measure transformation from the original model. For any given $v\in L^\infty([0,T]\times\t;\r^2)$, take $$Z^v(t):=\exp\left[\sum_{i=1}^n\int_{0}^{t}\frac{ v(s,X_i(s))}{\sqrt{2\nu}}dB_i(s)-\frac{n}{4\nu}\int_{0}^t \int_{\t}| v(s)|^2d\rho_n(s)ds\right].$$
By Proposition 5.12 of \cite{RN245}, it's a martingale. Thus we can apply Girsanov formula. Let $\mathcal{F}_T$ be the augmented filtration given by initial data and the Brownian motion (see (2.3) in Section 5.2 of \cite{RN245}). For $A\in \mathcal{F}_T$, let $P^v(A)(=P^v_T(A)):=\e [\chi_{A}Z^v(T)].$ Let $W_i(t)=B_i(t)-\frac{v(t,X_i(t))}{\sqrt{2\nu}}$ and then $\{W_t\}_{0\leq t\leq T}$ is a $2n$-dimensional Brownian motion under $P^v$. In addition, we have
\begin{equation}\label{equ_sde1}
dX_i(t)=\frac{1}{n}\sum_{i\neq j}\k(X_i-X_j)dt+ v(t,X_i)dt+\sqrt{2\nu}dW_i(t).
\end{equation}
Formally, the empirical distribution $\rho_n$  of (\ref{equ_sde1}) would converge to the solution of
\begin{equation}\label{def_weaks}
\partial_t \rho-\nu\Delta \rho+\div(\op(\rho))+\div(\rho v)=0.
\end{equation}
Hence a definition of weak solution is required for characterizing the mean-field limit.
\begin{defn}\label{def_2_1}
For $v\in L^\infty([0,T]\times\t;\r^2)$, we say $\rho\in C([0,T];\p(\t))$ is a weak solution of (\ref{def_weaks})
if $Q_T(\rho)<\infty$ and for each $\phi\in C^\infty([0,T]\times \t)$ and $0\leq s<t\leq T$,
$$
\begin{aligned}
&\lr \phi(t),\rho(t)\rr-\lr \phi(s),\rho(s)\rr=\int_{s}^{t}\lr \partial_t \phi(r),\rho(r)\rr dr
+\nu\int_{s}^{t}\lr \Delta \phi(r),\rho(r)\rr dr\\
&+\int_{s}^{t}\lr \nabla \phi(r),\op(\rho(r))\rr dr+\int_{s}^{t}\int_{\t}\nabla \phi(r)\cdot  v(r)d\rho(r)dr.
\end{aligned}
$$
\end{defn}

Then we will prove the law of large numbers under $P^v$.
\begin{lem}\label{lem_meanfield}
Let $\gamma\in\p(\t)$ with $e(\gamma)<\infty$ and $ v\in L^\infty([0,T]\times\t;\r^2)$. Then there exists an unique weak solution $\rho^{\gamma,v}$ of (\ref{def_weaks}) such that $\rho(0)=\gamma$ and $$Q_T(\rho^{\gamma,v})\leq e(\gamma)+C_v T,
$$
where $C_v$ is a $v$-dependent constant.
In addition, for each $\varepsilon>0,R>0$, if a sequence $(\gamma_n)$ satisfies $\gamma_n\in \mathcal{X}_n$, $e(\g_n*\gamma_n)\leq R$ and $\lim_{n\to\infty}d(\gamma_n,\gamma)=0$, then
$$
\lim_{n\to \infty}P^v_{\gamma_n}\left(\sup_{0\leq t\leq T} d\left(\rho_n(t),\rho^{\gamma,v}(t)\right)>\varepsilon\right)= 0.
$$
\end{lem}
\begin{rem}
It might happen that $\mathcal{X}_n\cap\{\eta:e(\g_n*\eta)\leq R\}=\emptyset$ for small $n$, but it would not happen when $n,R$ big enough (see Lemma \ref{lem_mfd1}). Hence, we allow $\gamma_n$ to be defined only for large $n$.
\end{rem}

Generally speaking, this result would imply that if the initial data $\gamma_n$ converging to $\gamma$ in weak topology, then the limit of relative entropy of $P^v_T$ with respect to $P_T$ under the scaling $1/n$ would provide a large deviation lower bound, i.e.
$$
\liminf_{n\to\infty}\frac{1}{n}\log P(\rho\in B)\geq-\lim_{n\to\infty}\frac{1}{n}H\left(P^{v}_{\gamma_n}|P_{\gamma_n}\right)=-\frac{1}{4\nu}\int_{0}^T\int_{\t}|v(t)|^2\rho^{\gamma,v}(t)dt
$$
for any open set $B$ containing $\rho^{\gamma,v}$,
in which
$$H\left(P^{v}_{\gamma_n}|P_{\gamma_n}\right):=\e^v_{\gamma_n}\left(\log \frac{dP^v_{\gamma_n}}{dP_{\gamma_n}}\right)
=\e^v_{\gamma_n}\left(\log Z^v(T)\right)=\frac{n}{4\nu}\e^v_{\gamma_n}\left[\int_0^T\int_{\t} |v(t)|^2 d\rho_n(t)dt\right].$$
Moreover, at least formally, by taking the infimum for $v$ we will obtain exactly the rate function, i.e.
\begin{equation}\label{equ_infi}
\act_T(\rho)=\inf_{v':\rho=\rho^{\gamma,v'}}\frac{1}{4\nu}\int_{0}^T\int_{\t}|v'(t)|^2\rho(t)dt.
\end{equation}

Define $$\mathcal{J}=\{v=\nabla p: p:[0,T]\times \t\mapsto \r, p(t)\in C^2(\t),\|\nabla p\|_{\infty}+ \|\nabla^2 p\|_{\infty}<\infty\}.$$
In our paper, for $\rho$ regular enough, i.e. $\rho=\rho^{\gamma,v}$ for certain $v\in\mathcal{J}$, the infimum in (\ref{equ_infi}) is always taken in the case that $v'=v$.

Since our initial data is more general, we need a stronger result stated below involving the stability for initial values.
\begin{lem}\label{lem_lmfd}
Let $\gamma\in\p(\t)$ with $e(\gamma)<\infty$ and $v\in L^\infty([0,T]\times\t;\r^2)$. For any $\varepsilon>0,R>0,$ there exists $\delta \in (0,\varepsilon)$ such that for each sequence $(\gamma_n)$ satisfying $\gamma_n\in B_\delta(\gamma)\cap\mathcal{X}_n$ and $e(\g_n*\gamma_n)\leq R$,
\begin{equation}\label{equ_mfd}
\lim_{n\to \infty}P^v_{\gamma_n}\left(\sup_{0\leq t\leq T} d(\rho_n(t),\rho^{\gamma,v}(t))>\varepsilon\right)= 0,
\end{equation}
where $B_\delta(\gamma)=\{\eta\in \p(\t):d(\eta,\gamma)<\delta\}.$ In addition, if $v\in \mathcal{J}$, then there exists a constant $C_v$ dependent on $v$ such that
\begin{equation}\label{equ_mfd1}
\limsup_{n\to\infty}\left|\frac{1}{4\nu}\e^v_{\gamma_n}\left[\int_0^T \int_{\t} |v(t)|^2 d\rho_n(t)dt\right]
-\act_T(\rho^{\gamma,v})\right|\leq C_v\varepsilon.
\end{equation}
\end{lem}

As one may expect, the set of "nice" trajectories, consisting of all $\rho^{\gamma,v}$ for $v\in \mathcal{J}$, plays a key role in proof of lower bound of LDP. We want to prove each $\rho$ with finite rate function can be approximated by a series of $\rho^{\gamma,v},v\in \mathcal{J}$ in the sense stated below.
\begin{lem}[Density of nice trajectory]\label{lem_density}
For each $\gamma\in \p(\t)$ with $e(\gamma)<\infty$, define $F^{\gamma}_{reg}=\{\rho\in AC((0,T);\p(\t)):Q_T(\rho)<\infty,\rho(0)=\gamma\}.$ If $\rho\in F_{reg}^{\gamma}$ with $\act_T(\rho)<\infty$, there exists $v_n\in \mathcal{J}$ such that
$$
\lim_{n\to\infty}\sup_{0\leq t\leq T}d(\rho^{\gamma,v_n}(t),\rho(t))=0, \quad \limsup_{n\to\infty}\act_T(\rho^{\gamma,v_n})\leq \act_T(\rho).
$$
\end{lem}

Now we are ready to prove the large deviation lower bound. By classical result of large deviation principle (see \cite{RN150}, Proposition 1.15), to prove lower bound for open set, it suffices to prove that lower bound holds for any open ball $B_\varepsilon(\rho)$ in $C([0,T];\p(\t))$, which is stated below.
\begin{prop}[Lower bound]\label{prop_lowerbound}
Under the condition of Theorem \ref{thm_LDPplus}, for any $\rho \in C([0,T];\p(\t))$ and $\varepsilon>0$,
$$
\liminf_{n\to \infty}\frac{1}{n}\log P\left(\sup_{0\leq t\leq T} d(\rho_n(t),\rho(t))<\varepsilon \right)\geq -\I_T(\rho).
$$
\end{prop}
\begin{proof}
If $\I_T(\rho)=\infty$, the result is trivial. Otherwise, $\rho\in AC((0,T);\p(\t))$ and $Q_T(\rho)<\infty$. By Condition \ref{cond_ldp_2}, we can pick $R$ big enough such that \begin{equation}\label{equ_lower5}
\limsup_{n\to \infty}\frac{1}{n}\log P(e(\g_n*\rho_n(0))>R)<-\I_0(\gamma)-1.
\end{equation}
By LDP of initial data, for each $\delta>0$,
$$
\liminf_{n\to \infty}\frac{1}{n}\log P(\rho_n(0)\in B_{\delta}(\gamma))\geq -\mathcal{I}_0(\gamma),
$$
which in combination with (\ref{equ_lower5}) leads to
\begin{equation}\label{equ_lower10}
\liminf_{n\to \infty}\frac{1}{n}\log P(\rho_n(0)\in B_{\delta}(\gamma),e(\g_n*\rho_n(0))\leq R)\geq -\mathcal{I}_0(\gamma)>-\infty.
\end{equation}
Let $\mathcal{O}=B_\varepsilon (\rho),$ and assume $\rho(0)=\gamma.$
By Lemma \ref{lem_density}, we can select $v_m\in \mathcal{J}$ such that $\rho^{\gamma,v_m}\in B_{\varepsilon}(\rho)$ and
\begin{equation}\label{equ_lower2}
\lim_{m\to\infty}\sup_{0\leq t\leq T} d(\rho^{\gamma,v_m}(t),\rho(t))=0,\quad \limsup_{m\to\infty}\act_T(\rho^{\gamma,v_m})\leq\act_T(\rho).
\end{equation}
We claim that for each $m$,
\begin{equation}\label{equ_lower1}
\liminf_{n\to \infty}\frac{1}{n}\log P\left(\rho_n \in \mathcal{O}\right)\geq -\act_T(\rho^{\gamma,v_m})-\I_0(\gamma).
\end{equation}
Then by (\ref{equ_lower2}) we have $$
\liminf_{n\to \infty}\frac{1}{n}\log P\left(\rho_n \in \mathcal{O}\right)\geq -\act_T(\rho)-\I_0(\gamma)=-\I_T(\rho),
$$
and arrive at the lower bound.

So we just need to verify (\ref{equ_lower1}). For any small enough $\varepsilon_1<\varepsilon$ such that $\mathcal{O}_1=B_{\varepsilon_1} (\rho^{\gamma,v_m})\subset \mathcal{O}$, take $\delta<\varepsilon_1$ in Lemma \ref{lem_lmfd} such that (\ref{equ_mfd}) and (\ref{equ_mfd1}) holds replacing $\varepsilon$ by $\varepsilon_1$. Note that
\begin{equation}\label{equ_lowerm1}
\begin{aligned}
 P(\rho_n\in \mathcal{O}&)\geq \inf_{\gamma_n\in B_\delta(\gamma)\cap \mathcal{X}_n,e(\g_n*\gamma_n)\leq R}P_{\gamma_n}(\rho_n\in \mathcal{O})\\
 &\cdot P(\rho_n(0)\in B_{\delta}(\gamma),e(\g_n*\rho_n(0))\leq R).
 \end{aligned}
\end{equation}
By (\ref{equ_lower10}), there exists $n_0$ such that when $B_\delta(\gamma)\cap \mathcal{X}_n\neq \emptyset$ for $n\geq n_0$. Hence for $n\geq n_0$, take $\overline{\gamma}_n\in B_\delta(\gamma)\cap \mathcal{X}_n$ such that $e(\g_n*\overline{\gamma}_n)\leq R$ and
\begin{equation}\label{equ_lower6}
\frac{1}{n}\log P_{\overline{\gamma}_n}(\rho_n\in \mathcal{O})\leq
\inf_{\gamma_n\in B_\delta(\gamma)\cap \mathcal{X}_n,e(\g_n*\gamma_n)\leq R}
 \frac{1}{n}
\log P_{\gamma_n}(\rho_n\in \mathcal{O})+\frac{1}{n}.
\end{equation}
In view of (\ref{equ_lower10}), (\ref{equ_lowerm1}) and (\ref{equ_lower6}),
\begin{equation}\label{equ_lower3}
\begin{aligned}
&\liminf_{n\to \infty}\frac{1}{n}\log P(\rho_n\in \mathcal{O})\geq \liminf_{n\to \infty}\frac{1}{n}\log P_{\overline{\gamma}_n}(\rho_n\in \mathcal{O})\\
&+\liminf_{n\to \infty}\frac{1}{n}\log P(\rho_n(0)\in B_{\delta}(\gamma),e(\g_n*\rho_n(0))\leq R)\\
&\geq \liminf_{n\to \infty}\frac{1}{n}\log P_{\overline{\gamma}_n}(\rho_n\in \mathcal{O})-\I_0(\gamma).
\end{aligned}
\end{equation}
Since $\mathcal{O}_1\in \mathcal{F}_T$,
$$
P_{\overline{\gamma}_n}(\rho_n\in \mathcal{O}_1)=P_{\overline{\gamma}_n}^{v_m}(\mathcal{O}_1)\e_{\overline{\gamma}_n}^{v_m}\bigg[ \left(Z^{v_m}(T)\right)^{-1} \frac{\chi_{\mathcal{O}_1}}{P_{\overline{\gamma}_n}^{v_m}(\mathcal{O}_1)}\bigg].
$$
By Jensen's inequality,
$$
\begin{aligned}
&\frac{1}{n}\log P_{\overline{\gamma}_n}(\rho_n\in \mathcal{O}_1)\geq  -(P_{\overline{\gamma}_n}^{v_m}(\rho_n \in \mathcal{O}_1))^{-1}E_{\overline{\gamma}_n}^{v_m}\bigg[\frac{1}{n}\log Z^{v_m}(T);\mathcal{O}_1\bigg]\\
&+ \frac{1}{n}\log P_{\overline{\gamma}_n}^{v_m}(\rho_n\in \mathcal{O}_1).
\end{aligned}
$$
By Lemma \ref{lem_lmfd}, the second expression on the RHS of the above inequality converges to $0$. The first one is equal to
\begin{equation}\label{equ_lower7}
-\frac{1}{P_{\overline{\gamma}_n}^{v_m}(\rho_n \in \mathcal{O}_1)}\bigg\{
E_{\overline{\gamma}_n}^{v_m}\bigg[\frac{1}{n}\log Z^{v_m}(T)\bigg]-E_{\overline{\gamma}_n}^{v_m}\bigg[\frac{1}{n}\log Z^{v_m}(T);\mathcal{O}_1^c\bigg]
\bigg\}.
\end{equation}

Noting
$$
\begin{aligned}
&Z^{v_m}(T)=\exp\left[\sum_{i=1}^n\int_{0}^{T}\frac{ v_m(s,X_i(t))}{\sqrt{2\nu}}dW_i(t)+\frac{n}{4\nu}\int_{0}^T \int_{\t}| v_m(t)|^2d\rho_n(t)dt\right],\\
&\text{Var}^{v_m}_{\overline{\gamma}_n}\left[\frac{1}{n}\sum_{i=1}^n\int_{0}^{T}\frac{ v_m(t,X_i(t))}{\sqrt{2\nu}}dW_i(t)\right]\leq\frac{1}{2n\nu}\| v_m\|^2_{\infty}T,\\
&\bigg|\frac{1}{4\nu}\int_{0}^T \int_{\t}| v_m(t)|^2d\rho_n(t)\bigg|\leq\frac{1}{4\nu}\| v_m\|^2_{\infty}T,
\end{aligned}
$$
and by (\ref{equ_mfd}), the second term in (\ref{equ_lower7}) converges to zero, and the denominator in the first terms converges to one, so we have
$$
\lim_{n\to\infty}\bigg|(\ref{equ_lower7})+\e_{\overline{\gamma}_n}^{v_m}\bigg[\frac{1}{4\nu}\int_{0}^{T}\int_{\t}| v_m(t)|^2 d\rho_n(t)dt\bigg]\bigg| =0.
$$
Therefore, so far we have,
\begin{equation}
\begin{aligned}
&
\liminf_{n\to \infty}\frac{1}{n}\log P_{\overline{\gamma}_n}(\rho_n\in \mathcal{O})\geq \liminf_{n\to \infty}\frac{1}{n}\log P_{\overline{\gamma}_n}(\rho_n\in \mathcal{O}_1)\\
&\geq-\limsup_{n\to\infty} \e_{\overline{\gamma}_n}^{v_m}\bigg[\frac{1}{4\nu}\int_{0}^{t}\int_{\t}| v_m(t)|^2 d\rho_n(t)\bigg].
\end{aligned}
\end{equation}
Then by (\ref{equ_mfd1}),
\begin{equation}\label{equ_lower4}
\liminf_{n\to \infty}\frac{1}{n}\log P_{\overline{\gamma}_n}(\rho_n\in \mathcal{O})\geq -\act_T(\rho^{\gamma,p_m})-C'_{v_m} \varepsilon_1.
\end{equation}
So (\ref{equ_lower1}) follows from (\ref{equ_lower3}) and (\ref{equ_lower4}) by taking $\varepsilon_1\to 0$.
\end{proof}

\section{Prior energy estimation and crucial inequalities}\label{sec_ener}
This section investigates some estimations related to the energy functional along with $\op(\rho)$. In Section \ref{sec_molg}, we prove Condition \ref{cond_ldp_1} implies Condition \ref{cond_ldp_2}. Section \ref{sec_conest} contains some crucial inequalities, which will be used throughout the paper. The proof of Lemma \ref{lem_Q1} is provided in Section \ref{sec_prfQ1}.
\subsection{Energy with mollification}\label{sec_molg}
Recall for $x\in \r,$ $$\g(x)=\left\{
\begin{aligned}
&Ce^{-\frac{1}{1-4|x|^2}},& |x|<\frac{1}{2},\\
&0,& |x|\geq \frac{1}{2},
\end{aligned}\right.
$$
and $\g_n:\r^2\mapsto \r,$
\begin{equation}\label{def_gn}
\g_n(x)=m_n^2\g(m_n|x|),
\end{equation}
where $m_n\uparrow \infty,nm_n^{-2}\to \infty$. For convenience, we take $m_1\geq 5$.
In fact, the estimations in this section hold for each $\g$ supported on $\left[-\frac{1}{2},\frac{1}{2}\right]$ and satisfying the Condition 3.1 below.
\begin{cond}\label{cond_molg}
{\rm(1)} $\g$ is a non-negative smooth even function.\\
{\rm(2)} For each $x>0,$ $\g'(x)\leq 0$.\\
{\rm(3)} $\int_{0}^{\infty}2\pi r \g(r)dr=1.$\\
{\rm(4)} There exists a constant $C_{\g}$ such that
$\g(r)\leq -C_{\g}\g'(r).$
\end{cond}

Define $G_n:=\g_n*\g_n$, which also are smooth mollifiers. It is straightforward to show that $G_n(x)=m_n^2G(m_n|x|)$ for certain function $G$ satisfying Condition \ref{cond_molg} and being supported on $[-1,1]$, and
\begin{equation}\label{mollifier-property1}
    (G_n*f)(x)=\int_{0}^{\frac{1}{m_n}}m_n^2G(m_nr)\left[\int_{\partial B_r(x)} fdS\right]dr.
\end{equation}

\begin{lem}
Condition \ref{cond_ldp_1} implies Condtion \ref{cond_ldp_2}.
\end{lem}
\begin{proof}
Let $$F(r,x)=\frac{1}{2\pi r}\int_{\partial B_r(x)}\n(y) dS.$$
Then for $0<r<\frac{1}{2}$ and $x=(x_1,x_2)\in \left[-\frac{1}{2},\frac{1}{2}\right)^2,|x|\neq r$,
\begin{equation}\label{equ_parF}
\begin{aligned}
&\partial_r F(r,x)=
\frac{1}{2\pi}\int_0^{2\pi}\partial_r\n(x_1+r\cos \theta,x_2+r \sin \theta)d\theta\\
&=\frac{1}{2\pi}\int_0^{2\pi}\nabla \n(x_1+r\cos \theta,x_2+r \sin \theta)\cdot (\cos \theta,\sin \theta)d\theta
\\
&=\frac{1}{2\pi r}\int_{\partial B_{r}(x)}\nabla \n(y)\cdot \vec{n}dS=
\frac{1}{2\pi r}\int_{B_{r}(x)}\Delta \n(x)dx=\frac{r}{2}-\frac{1}{2\pi r}\chi_{r>|x|},
\end{aligned}
\end{equation}
where we used (\ref{equ_green}). It's easy to check $\lim_{r\to 0+}F(r,x)=\n(x)$. Hence, for $x\in \left[-\frac{1}{2},\frac{1}{2}\right)^2$ and $r\in [0,\frac{1}{2})$, integrate (\ref{equ_parF}) along $[0,r]$, then we have
$$
F(r,x)=\n(x)+\frac{r^2}{4}-\frac{1}{2\pi}\max\{\log(r)-\log(|x|),0).
$$
By (\ref{mollifier-property1}),
\begin{equation}\label{equ_gnn1}
\begin{aligned}
&(G_n*\n)(x)-\n(x)
=\int_{0}^{\frac{1}{m_n}}2\pi r m_n^2G(m_nr)\big[F(r,x)-\n(x)\big]dr\\
&=-\int_{0}^{\frac{1}{m_n}} r m_n^2G(m_nr)\max\{\log(r)-\log(|x|),0)dr+\frac{1}{2}
\int_{0}^{\frac{1}{m_n}}\pi r^3 m_n^2G(m_nr)dr.
\end{aligned}
\end{equation}

Since $\int_{0}^{\frac{1}{m_n}}\pi r^3 m_n^2G(m_nr)dr=m_n^{-2}\int_{0}^{1}\pi r^3G(r)dr$, then there exists a constant $C$ such that
\begin{equation}\label{equ_gnn3}
(G_n*\n)(x)\leq \n(x)+\frac{C}{m_n^2},\quad \forall x\notin\mathbb{Z}^2.
\end{equation}
By (\ref{equ_N}), for $0<r<\frac{1}{2}$, $F(r,0)\leq -\frac{1}{2\pi}\log r+C$ is well defined. Then also due to (\ref{mollifier-property1}),
$$
\begin{aligned}
&(G_n*\n)(0)=\int_{0}^{\frac{1}{m_n}}2\pi r m_n^2 G(m_n r)F(r,0)dr\leq  \frac{1}{2\pi}\log(m_n)+C.
\end{aligned}
$$
Hence, $$
\begin{aligned}
&e(\g_n*\rho_n(0))=\frac{1}{2}\int_{\t}(G_n*\n)(x-y)\rho_n(0,dx)\rho_n(0,dy)\\
&=\frac{1}{2n^2}\sum_{i,j=1}^n(G_n*\n)(X_i(0)-X_j(0))\\
&=\frac{1}{2n}(G_n*\n)(0)+\frac{1}{2n^2}\sum_{i\neq j}^n(G_n*\n)(X_i(0)-X_j(0))\\
&\leq \frac{1}{2n}(G_n*\n)(0)+\frac{1}{2n^2}\sum_{i\neq j}^n\n(X_i(0)-X_j(0))+\frac{C}{2m_n^2}\quad (\text{by  (\ref{equ_gnn3})})\\
&= \frac{1}{2n}(G_n*\n)(0)+e_0(\rho_n(0))+\frac{C}{2m_n^2}
\leq e_0(\rho_n(0))+\frac{\log (m_n)}{2\pi n}+\frac{C}{2n}+\frac{C}{2m_n^2},
\end{aligned}
$$
and the conclusion follows.
\end{proof}

\subsection{Convolution estimations}\label{sec_conest}

\begin{lem}\label{lemm1mf}
There exists a constant $C_0$ such that for $x\in \left[-\frac{1}{2},\frac{1}{2}\right)^2$, $|x||(G_n*\k)(x)- \k(x)|\leq C_0G(m_n|x|).$
\end{lem}
\begin{proof}
Take the gradient of (\ref{equ_gnn1}), and we have
$$
\nabla (G_n*\n)(x)-\nabla \n(x)
=\frac{x}{|x|^2}\int_{\min\{|x|,\frac{1}{m_n}\}}^{\frac{1}{m_n}}m_n^2 r G(m_nr)dr.$$
So
\begin{equation}\label{equ_gnn2}
|x|\big[\nabla (G_n*\n)(x)-\nabla \n(x)\big]
=\frac{x}{|x|}\int_{\min\{m_n|x|,1\}}^{1} r G(r)dr.
\end{equation}
By (4) of Condition \ref{cond_molg}, there exists $C_G>0$ such that
$$
\begin{aligned}
&\int_{s}^{1} r G(r)dr
\leq C_{G} \int_{s}^{1} \left[-r G'(r)-G(r)\right] dr+C_{G} \int_{s}^{1} G(r)dr\\
&\leq C_{G}sG(s)-C_{G}^2 \int_{s}^{1} G'(r)dr\leq\left(C_{G}+C_{G}^2\right)G(s),
\end{aligned}
$$
which means there exists $C_0$ such that $\int_{s}^{1} r G(r)dr\leq C_0G(s).$ By (\ref{equ_gnn2}) and noting that $\k=-\nabla^{\perp} \n$, we conclude the proof.
\end{proof}

Similar to the way of obtaining (\ref{equ_rucv}), for $\varphi\in C^1(\t;\r^2)$,
\begin{equation}\label{equ_ce0}
\lr \varphi,\gamma (G_n*\k*\gamma)\rr=\frac{1}{2}\int_{(\t)^2\backslash \d}[\varphi(x)-\varphi(y)]\cdot (G_n*\k)(x-y) \gamma(dx)\gamma(dy),
\end{equation}
so we can estimate the difference between  $\op(\gamma)$ and $\gamma (G_n*\k*\gamma)$.

\begin{lem}\label{lem_rzq2}
For any $R>0$, there exists a sequence $c_n\downarrow 0$ such that for each $\gamma\in \p(\t)$, satisfying $e(\gamma)\leq R<\infty$,
$$\big|\lr \nabla G_n*\n*\gamma, \op(\gamma)\rr\big|\leq
c_n\|\g_n*\gamma\|_{2}^2.$$
\end{lem}

\begin{proof}
Step 1.\\
First, we firstly prove that for each $\eta,\gamma\in \p(\t)$, there exists a constant $C_1$ not dependent on $\eta,\gamma,k$ and $n$ such that for $1\leq k\leq \frac{m_n-3}{2}$,
\begin{equation}\label{equ_105}
\begin{aligned}
&\big|\lr G_n*\n*\gamma,\div \op(\eta)-\div[\eta (G_n*\k*\eta)]\rr\big|\\
&\leq C_1\bigg[\frac{1}{k^2}+\sup_{z}\gamma\left(B_{\frac{k+1}{m_n}}(z)\right)\bigg]\|\g_n*\eta\|_{2}^2.
\end{aligned}
\end{equation}
Noticing that $G_n(x)=0$ if $|x|\geq\frac{1}{m_n}$ and
$$\int_{(\t)^2} G_n(x-y)\eta(dx)\eta(dy)=\|\g_n*\eta\|_{2}^2,$$
together with (\ref{equ_rucv1}), (\ref{equ_ce0}) and Lemma \ref{lemm1mf}, it's enough to show there exists a constant $C_1$ such that for each $x\neq y,$ $r(x,y)<\frac{1}{m_n}$,
\begin{equation}\label{equ_rzq1_1}
\frac{|\nabla (G_n*\n*\gamma)(y)-\nabla (G_n*\n*\gamma)(x)|}{r(x,y)}\leq C_1m_n^2\bigg[\frac{1}{k^2}+\sup_{z}\gamma\left(B_{\frac{k+1}{m_n}}(z)\right)\bigg].
\end{equation}
Since $\n$ can be seen as a periodic function, without loss of generality, we assume $x_i< y_i<x_i+\frac{1}{2}$ $(i=1,2)$, so that $r(x,y)=|x-y|$.
Let $z_0=\frac{x+y}{2}$. Then
\begin{equation}\label{equ_rzq1_2}
\begin{aligned}
&\frac{|\nabla (G_n*\n*\gamma)(y)-\nabla (G_n*\n*\gamma)(x)|}{|x-y|}\\
& \leq  \int_{B_{\frac{k+1}{m_n}}(z_0)}\frac{|(G_n*\nabla \n)(y-z)-(G_n*\nabla \n)(x-z)|}{|x-y|}\gamma(dz)\\
& +\int_{B_{\frac{k+1}{m_n}}(z_0)^c}\frac{|(G_n*\nabla \n)(y-z)-(G_n*\nabla \n)(x-z)|}{|x-y|}\gamma(dz)\\
&\leq  \sup_{z\in B_{\frac{k+1}{m_n}}(z_0)} \frac{|(G_n*\nabla \n)(y-z)-(G_n*\nabla \n)(x-z)|}{|x-y|} \gamma\left(B_{\frac{k+1}{m_n}}(z_0)\right)\\
& +\sup_{z\notin B_{\frac{k+1}{m_n}}(z_0)} \frac{|\nabla \n(y-z)-\nabla \n(x-z)|}{|x-y|},
\end{aligned}
\end{equation}
where we used (\ref{equ_gnn2}) to obtain $$
\nabla \n(x)=G_n*\nabla \n(x),\text{ if } r(x,\mathbb{Z}^2)>\frac{1}{m_n}.
$$
As mentioned in \cite{RN242}, $\partial_i \n\in \dot{W}^{-1,\infty}(\t)$, i.e. there exist $A_{i,j}\in L^\infty(\t)(i,j=1,2)$ such that $\partial_i\n=\sum_{j=1,2}\partial_j A_{i,j}$. Therefore,
$$[(G_n*\nabla \n)(y-z)-(G_n*\nabla \n)(x-z)]_i=
\sum_{j}\left[\partial_j G_n* A_{i,j}(y-z)-\partial_j G_n* A_{i,j}(x-z)\right],
$$
So $$
\begin{aligned}
&\frac{|(G_n*\nabla \n)_i(y-z)-(G_n*\nabla \n)_i(x-z)|}{|x-y|}\\
&\leq \int_{B_{\frac{2}{m_n}}(z_0-z)}\frac{
\sum_{j}|\partial_j G_n(y-z-w)-\partial_j G_n(x-z-w)|
}{|x-y|}A_{i,j}(w)dw\\
&\leq \frac{C \sup |\nabla^2 G_n|}{m_n^2}\leq Cm_n^2.
\end{aligned}
$$
Since $|x-z_0|=|y-z_0|<\frac{1}{2m_n}$, for $z\in B^c_{\frac{k+1}{m_n}}(z_0)$ we have
$$
x,y\in B_{\frac{k+3/2}{m_n}}(z)\backslash \overline{B}_{\frac{k+1/2}{m_n}}(z)\subset  B_{1/2}(z)\backslash \overline{B}_{\frac{k}{m_n}}(z),
$$
and thus by (\ref{equ_N}),
$$
|\nabla \n(y-z)-\nabla \n(x-z)|\leq |x-y|\sup_{\frac{1}{2}>|w|>\frac{k}{m_n}}|\nabla^2\n(w)|\leq C\frac{m_n^2}{k^2}|x-y|.
$$
Therefore (\ref{equ_rzq1_2}) can deduce (\ref{equ_rzq1_1}) and we arrive at (\ref{equ_105}).\\
Step 2.\\
Take $\eta=\gamma$ in (\ref{equ_105}). Noting that $(G_n*\k*\gamma)\cdot (G_n*\nabla \n*\gamma)=0$, for $1\leq k\leq \frac{m_n-3}{2}$, we have
$$
\big|\lr \nabla G_n*\n*\gamma, \op(\gamma)\rr\big|\leq C_1\bigg[\frac{1}{k^2}+\sup_{z}\gamma\left(B_{\frac{k+1}{m_n}}(z)\right)\bigg]\|\g_n*\gamma\|_{2}^2.
$$
By (\ref{equ_pest1}),
$$
\sup_{z}\gamma\left(B_{\frac{k+1}{m_n}}(z)\right)\leq \bigg(\frac{C_{\n}+4\pi e(\gamma)}{\log(m_n/2)-\log(k+1)}\bigg)^{\frac{1}{2}}.
$$
Taking
$$
c_n=C_1\min_{1\leq k<\frac{m_n-3}{2}}\left[\frac{1}{k^2}
+\bigg(\frac{C_{\n}+4\pi R}{\log(m_n/2)-\log(k+1)}\bigg)^{\frac{1}{2}}\right],
$$
then $$\big|\lr \nabla G_n*\n*\gamma,\op(\gamma)\rr\big|\leq
c_n\|\g_n*\gamma\|_{2}^2.$$
Through a simple calculation, we can check $\lim_{n\to\infty}c_n=0.$\\
\end{proof}
\begin{rem}
The rate for $m_n\to\infty$ is not needed in the proof of Lemmas \ref{lemm1mf} and \ref{lem_rzq2}, i.e., we haven't used the fact that $nm_n^{-2}\to\infty$.
\end{rem}

We also need a generalised version of Ladyzhenskaya’s inequality, which is often used to study two-dimensional Navier-Stokes equation.
\begin{lem}\label{lem_B_6}
There exists $C_1,C_2$ such that
for each $\gamma,\eta\in \p(\t)\cap L^2(\t)$,
$$\|\k*(\gamma-\eta)\|_{4}^2=\|\nabla \n*(\gamma-\eta)\|_{4}^2\leq C_1\|\gamma-\eta\|_{2},$$
$$\int_{\t}|\nabla\n*\gamma|^2d\gamma=\int_{\t}|\k*\gamma|^2d\gamma\leq C_2 \|\gamma-1\|_{2}^2.$$
\end{lem}
\begin{proof}
Let $\phi\in C^\infty_c(\r^2)$ be a radial function such that $0\leq \phi\leq 1,\phi(x)=1$ for $|x|\leq \frac{1}{2}$ and $\phi(x)=0$ for $|x|\geq 1.$ We define for $\frac{1}{2}>B>0,$ $\phi_B(x)=\phi(x/B)$, and $\k_{1,B}=\phi_B\k,\k_{2,B}=(1-\phi_B)\k.$
We start with proving that for $1\leq p<2, q>2,$ there exist constants $C_{q},C'_{p}$, such that
\begin{equation}\label{equ_K12}
\|\k_{2,B}\|_{q}\leq C_{q}B^{\frac{2}{q}-1},\quad \|\k_{1,B}\|_{p}\leq C'_{p}B^{\frac{2}{p}-1}.
\end{equation}

By (\ref{equ_N}), we can find $C_0>0$ such that for each $x\in B_{\frac{1}{2}}((0,0))\backslash \{(0,0)\},$
\begin{equation}\label{equ_lem_KB1}
|\k(x)|\leq C_0|x|^{-1},\quad |\nabla \k(x)|\leq C_0 |x|^{-2}.
\end{equation}
So there exists $C_{q}$ such that
$$
\|\k_{2,B}\|_{q}
\leq C_0\left(\int_{|x|>\frac{B}{2}}|x|^{-q}dx\right)^{1/q} \leq C_{q} B^{2/q-1}.
$$
Also we have
$$
\|\k_{1,B}\|_{p}^p\leq C_0^p\int_{[-1/2,1/2]^2}\phi(x/B)^p |x|^{-p}dx
\leq C_0^pB^{2-p}\int_{|y|\leq 1}\phi(y)^p |y|^{-p}dy.
$$
Let $C'_{p}=C_0\left(\int_{|y|\leq 1} |y|^{-p}dy\right)^{\frac{1}{p}}$. Then we have
$$
\| \k_{1,B}\|_{p}\leq C'_{p}B^{\frac{2}{p}-1}.
$$

Turn to the proof of the desired inequalities. A consequence of (\ref{equ_K12}) along with Young's inequality implies that there exists a constant $C$, such that
$$
\|(\phi_B\k)*(\gamma-\eta)\|_{4}\leq \|\k_{1,B}\|_{{\frac{4}{3}}}\|\gamma-\eta\|_{2}\leq CB^{\frac{1}{2}}\|\gamma-\eta\|_{2},
$$
and
$$
\|[(1-\phi_B)\k]*(\gamma-
\eta)\|_{4}\leq  \|\k_{2,B}\|_{4}\|\gamma-\eta\|_{1}\leq C/B^{\frac{1}{2}}.
$$
If $\|\gamma-\eta\|_{2}>2$, we take $B=\|\gamma-\eta\|_{2}^{-1}$ and have $$\|\k*(\gamma-\eta)\|_{4}\leq
 \|(\phi_B\k)*(\gamma-\eta)\|_{4}+\|[(1-\phi_B)\k]*(\gamma-
\eta)\|_{4}\leq 2C\|\gamma-\eta\|_{2}^\frac{1}{2}.$$
If $\|\gamma-\eta\|_{2}\leq 2$, by Young's inequality,
$$\|\k*(\gamma-\eta)\|_{4}\leq \|\k\|_{{\frac{4}{3}}}\|\gamma-\eta\|_{2}\leq \sqrt{2} \|\k\|_{{\frac{4}{3}}} \|\gamma-\eta\|_{2}^\frac{1}{2},$$
in which $\|\k\|_{\frac{4}{3}}<\infty$ due to (\ref{equ_N}).
So we can pick $C_1=\max\left\{\sqrt{2}\|\k\|_{{\frac{4}{3}}},2C\right\}$ such that
$$\|\k*(\gamma-\eta)\|_{4}^2\leq C_1\|\gamma-\eta\|_{2}.$$
To obtain the second inequality, noting that $\int_{\t}\k(x)dx=0$ and by young's inequality, we have $$
\|\k*\gamma\|_{2}=\|\k*(\gamma-1)\|_{2}\leq \|\k\|_{1}\|\gamma-1\|_{2}.
$$
Hence by Holder's inequality,
$$
\begin{aligned}
&\int_{\t}|\k*\gamma|^2d\gamma= \int_{\t}|\k*\gamma|^2dx+\int_{\t}|\k*\gamma|^2d(\gamma-1)
\leq \|\k*\gamma\|_2^2+\|\k*\gamma\|_4^2\|\gamma-1\|_2\\
&\leq (\|\k\|_{1}^2+C_1)\|\gamma-1\|_{2}^2.
\end{aligned}
$$
\end{proof}

\begin{cor}\label{cor_B_6}
For each $\delta>0$, $\gamma\in \p(\t)$, smooth mollifier $J$ and $\gamma-$measurable function $\varphi$ satisfying $\int_{\t}|\varphi|^2d\gamma<\infty$, one has
$$
\left|\int_{\t}\varphi \cdot(J*J*\k*\gamma) d\gamma\right|\leq
\delta \|J*\gamma-1\|_{2}^2+\frac{C_1}{4\delta} \int_{\t}|\varphi|^2 d\gamma,
$$
where $C_1$ is the constant in Lemma \ref{lem_B_6}.
\end{cor}
\begin{proof}
By Jensen's inequality $$
|J*J* \k*\gamma|^2(x) \leq (J*|J* \k*\gamma|^2)(x).
$$
Then using $ab\leq \frac{\delta}{C_1}a^2+\frac{C_1}{4\delta}b^2$ and by Lemma \ref{lem_B_6}, we have $$
\begin{aligned}
&\left|\int_{\t}\varphi\cdot (J*J*\k*\gamma) d\gamma\right|\leq \frac{\delta}{C_1}\int_{\t} |J*J*\k*\gamma|^2 d\gamma
+\frac{C_1}{4\delta} \int_{\t}|\varphi|^2 d\gamma\\
&\leq \frac{\delta}{C_1} \int_{\t} |J*\k*\gamma|^2 d(J*\gamma)
+\frac{C_1}{4\delta} \int_{\t}|\varphi|^2 d\gamma\\
&\leq \delta \|J*\gamma-1\|_{2}^2
+\frac{C_1}{4\delta} \int_{\t}|\varphi|^2 d\gamma.
\end{aligned}
$$
\end{proof}

\subsection{Proof of Lemma \ref{lem_Q1}}\label{sec_prfQ1}
We need a generalization of the Doob submartingale inequality.
\begin{lem}\label{lem_ville}
If $M(t)$ is a positive continuous local martingale, then for each $l\in \r,$
$$
P\left(\sup_{0\leq t\leq T}\log M(t)\geq l\right)\leq \frac{\e M(0)}{e^l}.
$$
\end{lem}
\begin{proof}
Since $M(t)$ is a positive local martingale, it's a supermartingale.
Let $\tau=\inf\{t:M(t)>e^l\}\wedge T$. Then $M(t\wedge \tau)$ is a non-negative supermartingale. Hence
$$
e^lP\left(\sup_{0\leq t\leq T}M(t)\geq e^l\right)\leq \e(M(\tau)\chi_{M(\tau)\geq e^l})\leq \e(M(\tau))\leq \e (M(0)).
$$
\end{proof}

Now we are ready to give a quantitative version of Lemma \ref{lem_Q1}, which together with Condition \ref{cond_ldp_2} implies Lemma \ref{lem_Q1}.
\begin{lem}\label{lem_Q1plus}
There exists constants $\lambda>0$ such that for each sequence $(\eta_n)$ satisfying $\eta_n\in \x_n$ and $e(\g_n*\eta_n)\leq R,$
\begin{equation}\label{iequ_Q1plus}
\begin{aligned}
\limsup_{n\to\infty}&\frac{1}{n}\log P_{\eta_n}\bigg(\sup_{0<t\leq T} \bigg(e(\g_n*\rho_n(t))\\
&+ \frac{\nu}{2}\int_0^t \|\g_n*\rho_n(s)-1\|_{2}^2ds\bigg)
>l\bigg)\leq -\lambda(l-R)
\end{aligned}
\end{equation}
for any $l\in \r$.
\end{lem}
\begin{proof}
Define $\n_n:=G_n*\n$, $\omega_n:=\frac{\nu}{n}(G_n(0)-1)$. By Ito's formula and using the fact $-\Delta (G_n*\n)=G_n-1$, we have
$$
\begin{aligned}
&d\left[\frac{1}{2}\lr\n_n*\rho_n(t),\rho_n(t)\rr\right]=\frac{1}{2n^2}\sum_{i\neq j}d\n_n(X_i(t)-X_j(t))\\
&=\frac{1}{2n^2}\sum_{i\neq j}\nabla \n_n(X_i(t)-X_j(t))(dX_i(t)-dX_j(t))
+\frac{\nu}{n^2}\sum_{i\neq j}\Delta \n_n(X_i(t)-X_j(t))dt\\
&=\frac{1}{n^2}\sum_{i\neq j}\nabla \n_n(X_i(t)-X_j(t))dX_i(t)
+\frac{\nu}{n^2}\sum_{i\neq j}\Delta \n_n(X_i(t)-X_j(t))dt\\
&=\frac{1}{n^2}\sum_{i=1}^n\sum_{j=1}^n\sum_{k\neq i}\nabla \n_n(X_i(t)-X_j(t))\left(\k(X_i(t)-X_k(t))dt+\sqrt{2\nu}dB_i(t)\right)\\
&+\frac{\nu}{n^2}\sum_{i,j=1}^n \Delta \n_n(X_i(t)-X_j(t))dt
-\frac{\nu}{n}\Delta \n_n(0)dt\\
&=\lr \nabla \n_n*\rho_n(t),\op(\rho_n(t)) \rr dt-\nu\|\g_n*\rho_n(t)-1\|_{2}^2dt\\
&+\frac{\sqrt{2\nu}}{n^2}\sum_{i,j=1}^{n}\nabla \n_n(X_i(t)-X_j(t))dB_i(t)+\omega_n dt.
\end{aligned}
$$
After further calculation, for any $\lambda>0,$
\begin{equation}
\begin{aligned}
&\exp\bigg\{
n\lambda\bigg[e(\g_n*\rho_n(t))-e(\g_n*\rho_n(0))
+\nu\int_{0}^{t}\|\g_n*\rho_n(s)-1\|_{2}^2ds\\&-\int_{0}^{t}\lr\nabla \n_n*\rho_n(s),\op(\rho_n(s))\rr ds-\lambda\nu\int_{0}^{t} \lr|\nabla \n_n*\rho_n(s)|^2,\rho_n(s)\rr ds-\omega_nt
\bigg]
\bigg\}
\end{aligned}
\end{equation}
is a positive continuous martingale. By Lemma \ref{lem_ville}, for each $\eta_n\in \x_n$ with $e(\g_n*\eta_n)\leq R,$
\begin{equation}\label{equ_doob1plus}
\begin{aligned}
&P_{\eta_n}\bigg\{\sup_{0<t\leq T}\bigg[e(\g_n*\rho_n(t))-\int_{0}^{t}\lr\nabla \n_n*\rho_n(s),\op(\rho_n(s))\rr ds\\
&+\nu\int_{0}^{t}\|\g_n*\rho_n(s)-1\|_{2}^2ds-\nu \lambda\int_{0}^{t}\int_{\t}|\nabla \n_n*\rho_n(s)|^2\rho_n(s,dx)  ds\bigg]>l\bigg\}\\
&\leq e^{-n\lambda(l-R-\omega_nT)}.
\end{aligned}
\end{equation}
By Jensen inequality $$
|\nabla \n_n*\rho_n(t,x)|^2=|\g_n*\g_n*\nabla \n*\rho_n(t,x)|^2\leq \left(\g_n*|\g_n*\nabla \n*\rho_n(t,\cdot)|^2\right)(x),
$$
Then by Lemma \ref{lem_B_6}, there exists $C_2>0$ such that
\begin{equation}\label{equ_b6}
\int_{\t}|\nabla \n_n*\rho_n(t)|^2\rho_n(t,dx) \leq \int_{\t}|\nabla \n*\g_n*\rho_n(t)|^2(\g_n*\rho_n)(t,dx) \leq C_2\|\g_n*\rho_n(t)-1\|_{2}^2,
\end{equation}
Combining (\ref{equ_b6}) with (\ref{equ_doob1plus}) and taking $\lambda$ such that $C_2 \lambda=\frac{1}{3}$, we have
$$
\begin{aligned}
&P_{\eta_n}\bigg\{\sup_{0<t\leq T}\bigg(e(\g_n*\rho_n(t))-\int_{0}^{t}\lr\nabla \n_n*\rho_n(s),\op(\rho_n(s))\rr ds\\
&+\frac{2\nu}{3}\int_{0}^{t}\|\g_n*\rho_n(s)-1\|_{2}^2 ds\bigg)>l\bigg\}\leq e^{-n\lambda(l-R-\omega_nT)}.
\end{aligned}
$$
For each $\varepsilon>0$, define $E_{n,l,\varepsilon}$ as a subset of $C([0,T];\p(\t))$ by
$$
\begin{aligned}
&E_{n,l,\varepsilon}:=\bigg\{\rho:\sup_{0<t\leq T}\bigg(e(\g_n*\rho(t))-\int_{0}^{t}\lr\nabla \n_n*\rho(s),\op(\rho(s))\rr ds\\
&+\frac{2\nu}{3}\int_{0}^{t}\|\g_n*\rho(s)-1\|_{2}^2 ds\bigg)\leq l-\varepsilon\bigg\}.
\end{aligned}
$$
By Lemma \ref{lem_rzq2}, there exists a sequence $c_n\downarrow 0$ only dependent on $l$ such that if $e(\g_n*\rho(s))\leq l$, then
\begin{equation}\label{equ_lem_Q1plus_1}
\big|\lr \nabla\n_n*\rho(s), \op(\rho(s))\big|\rr\leq c_n\|\g_n*\rho(s)\|_{2}^2=c_n\|\g_n*\rho(s)-1\|_{2}^2+c_n.
\end{equation}
Take $n$ big enough such that $\frac{2\nu}{3}-c_n>\frac{\nu}{2}$ and $c_n<\varepsilon$. Define $$\tau_l:=\inf\left\{t:e(\g_n*\rho(t))+\frac{\nu}{2}\int_{0}^{t}\|\g_n*\rho(s)-1\|_{2}^2 ds>l\right\}\wedge T.$$ We claim that if $\rho\in E_{n,l,\varepsilon}$ and $e(\g_n*\rho(0))\leq l$, then $\tau_l=T$.

To prove it by contradiction, suppose $\tau_l<T$. Since $\rho \in C([0,T];\p(\t))$,
\begin{equation}\label{equ_lem_Q1plus_21}
e(\g_n*\rho(\tau_l))+\frac{\nu}{2}\int_{0}^{\tau_l}\|\g_n*\rho(s)-1\|_2^2ds=l.
\end{equation}
Noting $\rho\in E_{n,l,\varepsilon}$, we have
\begin{equation}\label{equ_lem_Q1plus_2}
e(\g_n*\rho(\tau_l))-\int_{0}^{\tau_l}\lr\nabla \n_n*\rho(s),\op(\rho(s))\rr ds+\frac{2\nu}{3}\int_{0}^{\tau_l}\|\g_n*\rho(s)-1\|_{2}^2 ds\leq l-\varepsilon.
\end{equation}
However, (\ref{equ_lem_Q1plus_21}), (\ref{equ_lem_Q1plus_1}) and (\ref{equ_lem_Q1plus_2}) imply $$
l-c_n\leq e(\g_n*\rho(\tau_l))+\left(\frac{2\nu}{3}-c_n\right)\int_{0}^{\tau_l}\|\g_n*\rho(s)-1\|_{2}^2 ds-c_n\leq l-\varepsilon,
$$
which is a contradiction.
Hence, for each $\rho\in E_{n,l,\varepsilon}$ with $e(\g_n*\rho(0))\leq l$, there exists $n_1$ such that if $n>n_1$ then $Q_T(\g_n*\rho_n)\leq l.$

Therefore, $$P_{\eta_n}\left(Q_T(\g_n*\rho_n)>l\right)\leq P_{\eta_n}\left(\g_n*\rho_n\in \left(E_{n,l,\varepsilon}\right)^c\right)\leq e^{-n\lambda(l-\varepsilon-R-\omega_nT)}.$$
Recall $nm_n^{-2}\to\infty,$ so $\limsup_{n\to\infty}\omega_n\leq \lim_{n\to\infty}\frac{\nu m_n^2\|G\|_{\infty}\|\n\|_{1}}{n}=0$. By the arbitrariness of $\varepsilon$ the conclusion follows.
\end{proof}

\section{Regularity of trajectories with finite rate function}\label{sec_regular}
In this section, we study the regularity of trajectories with finite rate function to give a more direct expression for the rate function and prove Lemma \ref{lem_vrotrf}. These regularity results are also preparations for the proof of subsequent lemmas in the next section.

\subsection{Weighted Sobolev space and Riesz representation}

To obtain the explicit form of rate function, we need a notation of weighted Sobolev space $H^{1}_\rho([0,T]\times\t).$
For $\phi\in C^\infty([0,T]\times\t)$ and $\rho\in C([0,T];\p(\t))$, define the norm
$$\|\phi\|_{1,\rho,T}^2=\int_{0}^{T}\int_{\t}|\nabla \phi(t)|^2d\rho(t)dt.$$
Define $H^{1}_\rho([0,T]\times\t)$ as the completion of $C_0^\infty([0,T]\times\t)=\{\phi\in C^\infty([0,T]\times\t):\int_0^T\int_{\t} \phi(t,x)dxdt=0\}$ under $\|\cdot\|_{1,\rho,T}.$
That is a Hilbert space with inner product
$$
\lr p_1,p_2\rr_{1,\rho,T}=\frac{1}{4}\left(\| p_1+p_2\|^2_{1,\rho,T}-\| p_1-p_2\|^2_{1,\rho,T}\right).
$$
Then for any $p\in H^{1}_\rho([0,T]\times\t)$, there exists a function $\hat{\nabla}p$ defined on $[0,T]\times\t$ such that
$$
\int_0^T \int_{\t}|\hat{\nabla}p(t)|^2d\rho(t)dt<\infty,
$$
and
$$
\lr p,\phi\rr_{1,\rho,T}=\int_{0}^{T}\int_{\t}\hat{\nabla}p(t,x)\cdot \nabla \phi(t,x)d\rho(t)dt,\quad \forall \phi\in C^\infty([0,T]\times\t).
$$
The inner product thus can be written as
$$
\lr p_1,p_2\rr_{1,\rho,T}=\int_{0}^{T}\int_{\t}\hat{\nabla}p_1(t,x)\cdot \hat{\nabla}p_2(t,x)d\rho(t)dt,\quad \forall p_1,p_2\in  H^{1}_\rho([0,T]\times\t).
$$
Recall that
$$
\begin{aligned}
&\overline{\act}_{T}(\rho):=\sup_{\phi\in C^\infty([0,T]\times \t)}\bigg(
\lr\phi(T),\rho(T)\rr-\lr\phi(0),\rho(t_0)\rr
-\int_{0}^{T}\lr\partial_t \phi(s),\rho(s)\rr ds\\
&-\nu\int_{0}^{T}\lr\Delta \phi(s),\rho(s)\rr ds
-\int_{0}^{T}\lr\nabla \phi(s),\op(\rho(s))\rr ds
-\nu\int_{0}^{T}\int_{\t}|\nabla \phi(s)|^2d\rho(s)ds\bigg).
\end{aligned}
$$
We claim that $\overline{\act}_{T}(\rho)<\infty$ means $\partial_t\rho-\nu\Delta \rho+\div \op(\rho)$ can be seen as a bounded linear operator on $H^{1}_\rho([0,T]\times\t)$. To obtain that, for $\phi\in C^\infty([0,T]\times \t)$ define
$$
\begin{aligned}
&L_\rho\phi:=
\lr\partial_t \rho-\nu\Delta \rho +\div \op(\rho),\phi\rr=\lr\phi(T),\rho(T)\rr-\lr\phi(0),\rho(0)\rr\\
&-\int_0^T \lr\partial_t \phi(t),\rho(t)\rr dt
-\nu\int_{0}^{T}\lr\Delta \phi(t),\rho(t)\rr dt
-\int_{0}^{T}\lr \nabla \phi(t),\op(\rho(t))\rr dt.
\end{aligned}
$$
By the definition of $\overline{\act}_T$, for a test function $\phi$, taking $\phi_k=k\phi$, we have
$$\sup_{k}\left(L_{\rho}(k\phi)-\nu\|k\phi\|_{1,\rho,T}^2\right)=\frac{(L_{\rho}\phi)^2}{4\nu \|\phi\|_{1,\rho,T}^2}\leq \overline{\act}_T(\rho),$$ implying $L_{\rho}$ is a bounded operator on $H^{1}_\rho([0,T]\times\t)$.

By Riesz representation theorem, if $\overline{\act}_{T}(\rho)<\infty$, there exists $p\in H^{1}_{\rho}([0,T]\times\t)$ such that $$
\int_{0}^{T}\int_{\t}|\hat{\nabla} p(t)|^2d\rho(t)dt\leq 4\nu\overline{\act}_T(\rho)$$ and $L_{\rho}\phi=\lr\phi,p\rr_{1,\rho,T}$, i.e. for each $\phi\in C^\infty([0,T]\times \t),$
\begin{equation}\label{equ_std_m1}
\begin{aligned}
&\lr \phi(T),\rho(T)\rr-\lr \phi(0),\rho(0)\rr=\int_{0}^{T}\lr\partial_t\phi(r),\rho(r)\rr dr +\nu\int_{s}^{t}\lr \Delta \phi(r),\rho(r)\rr dr\\
&+\int_{0}^{T}\lr\nabla \phi(r)\cdot \hat{\nabla}p(r),\rho(r)\rr dr+\int_{0}^{T}\lr \nabla \phi(r),\op(\rho(r))\rr dr.
\end{aligned}
\end{equation}
In addition, by (\ref{equ_std_m1}), taking smooth test function approximating $\frac{1}{2\nu}\hat{\nabla} p$ in the definition of $\overline{\act}_T(\rho)$, we finally have
\begin{equation}\label{equ_ovlact}
\overline{\act}_{T}(\rho)=\frac{1}{4\nu}\int_{0}^{T}\int_{\t}|\hat{\nabla} p(t)|^2d\rho(t)dt.
\end{equation}

More properties of weighted Sobolev space is included in Appendix A.

\subsection{Production estimations of energy and entropy}
In this subsetion we will prove production estimations of energy and entropy.
\begin{lem}\label{lem_etrest}
Suppose $\rho \in C([0,T];\p
(\t))$, $Q_T(\rho)<\infty$ and $\overline{\act}_T(\rho)<\infty.$ Then there exists $\hat{\nabla}p:[0,T]\times \t\mapsto \r,$ such that (\ref{equ_ovlact}) holds and for each $\phi\in C^\infty([0,T]\times \t)$,
\begin{equation}\label{equ_std_1}
\begin{aligned}
&\lr \phi(t),\rho(t)\rr-\lr \phi(s),\rho(s)\rr=\int_{s}^{t}\lr\partial_t\phi(r),\rho(r)\rr dr +\nu\int_{s}^{t}\lr \Delta \phi(r),\rho(r)\rr dr\\
&+\int_{s}^{t}\lr\nabla \phi(r)\cdot \hat{\nabla}p(r),\rho(r)\rr dr+\int_{s}^{t}\lr \nabla \phi(r),\op(\rho(r))\rr dr.
\end{aligned}
\end{equation}
In addition, for $0\leq s<t\leq T,$
\begin{equation}\label{equ_etrest1}
e(\rho(t))-e(\rho(s))=-\nu\int_{s}^{t}\|\rho(r)-1\|_{2}^2dr+\int_{s}^{t}\int_{\t}(\nabla\n*\rho)(r,x)\cdot
\hat{\nabla}p(r,x)\rho(r,x)dr.
\end{equation}
\end{lem}
\begin{proof}
Since $\overline{\act}_T(\rho)<\infty$, as discussed in the previous subsection, there exists $\hat{\nabla}p$ satisfying (\ref{equ_ovlact}) and (\ref{equ_std_m1}) holds for each $\phi\in C^\infty([0,T]\times\t)$. Since $\rho\in C([0,T];\p(\t))$, it's standard to see (\ref{equ_std_1}) holds for each $\phi\in C^\infty([0,T]\times\t)$ by smooth approximation of truncated functions only supported on $[s,t]$.

Take a smooth mollifier $J$ with compact support and define $J_n(x):=n^2J\left(\frac{x}{n}\right)$.  Taking $\phi(y)=J_n(x-y)$ in (\ref{equ_std_1}), noting $\rho(t)\in L^2(\t)$ for almost every $t$ and by (\ref{equ_rucv}), we can check that $\hat{\rho}_n(t,x):=(J_n*\rho)(t,x)$ is absolutely continuous with respect to $t$ and for a.e.-$t$
\begin{equation}\label{equ_std_3}
\partial_t \hat{\rho}_n(t,x)-\nu\Delta \hat{\rho}_n(t,x)+[J_n*\div (\rho(t) u(t))](x)+[J_n*\div (\rho(t) \hat{\nabla}p(t))](x)=0,
\end{equation}
where $u(t)=\k*\rho(t)$.

Now we show $e(\hat{\rho}_n(t))$ is absolutely continuous with respect to $t$. Notice that $\|\hat{\rho}_n\|_{\infty}$ and $\|\partial_t\hat{\rho}_n\|_{\infty}$ are bounded by constants which only depend on $J_n$, then for $0\leq s<t\leq T,$ $$
\begin{aligned}
&|e(\hat{\rho}_n(t))-e(\hat{\rho}_n(s))|\\
&=\frac{1}{2}\bigg|\int_{(\t)^2}\n(x-y)\hat{\rho}_n(t,x)\hat{\rho}_n(t,y)dxdy
-\int_{(\t)^2}\n(x-y)\hat{\rho}_n(s,x)\hat{\rho}_n(s,y)dxdy\bigg|\\
&=\frac{1}{2}\bigg|\int_{(\t)^2}\n(x-y)\big[\hat{\rho}_n(t,x)
(\hat{\rho}_n(t,y)-\hat{\rho}_n(s,y))
-(\hat{\rho}_n(t,x)-\hat{\rho}_n(s,x))\hat{\rho}_n(s,y)\big]
dxdy\bigg|\\
&\leq (t-s) \|\hat{\rho}_n\|_{\infty}\|\partial_t\hat{\rho}_n\|_{\infty} \int_{(\t)^2}|\n(x-y)|dxdy\leq C_n(t-s).
\end{aligned}
$$
Then by directly taking derivative and dominated convergence theorem, we have
$$
\partial_t e(\hat{\rho}_n(t,x))=\int_{(\t)^2}\n(x-y)\hat{\rho}_n(t,y)\partial_t \hat{\rho}_n(t,x)dxdy.
$$
By (\ref{equ_std_3}) and noticing $\Delta (J_n*J_n)*\n=1-J_n*J_n$ as well as $\int_{(\t)^2}(J_n*J_n)(x-y)\rho(t,dx)\rho(t,dy)-1=\|J_n*\rho(t)-1\|_{2}^2,$
\begin{equation}\label{equ_lem_etrest_1}
\begin{aligned}
&e(J_n*\rho(t))-e(J_n*\rho(s))=-\nu\int_{s}^{t}\|J_n*\rho(r)-1\|_{2}^2dr\\
&+\int_{s}^{t}\int_{\t}(J_n*J_n*\nabla \n*\rho)(r,x)\cdot u(r,x) \rho(r,x)dxdr\\
&+\int_{s}^{t}\int_{\t}(J_n*J_n*\nabla \n*\rho)(r,x)\cdot \hat{\nabla}p(r,x) \rho(r,dx)ds.
\end{aligned}
\end{equation}
To use (\ref{equ_lem_etrest_1}) to show (\ref{equ_etrest1}), we just need to prove that
\begin{equation}\label{equ_lem_etrest3}
\begin{aligned}
&\lim_{n\to\infty}\int_{s}^{t}\int_{\t}(J_n*J_n*\nabla \n*\rho)(r,x)\cdot u(r,x) \rho(r,x)dxdr=0,
\end{aligned}
\end{equation}
\begin{equation}\label{equ_lem_etrest4}
\begin{aligned}
&\lim_{n\to\infty}\int_{s}^{t}\int_{\t}(J_n*J_n*\nabla \n*\rho)(r,x)\cdot \hat{\nabla}p(r,x) \rho(r,x)dxdr\\
&=\int_{s}^{t}\int_{\t}(\nabla \n*\rho)(r,x)\cdot \hat{\nabla}p(r,x) \rho(r,x)dxdr,
\end{aligned}
\end{equation}
and the convergence of rest terms can be proved by Fatou's Lemma and Jensen's inequality.

Note that $\|\rho(r)\|_2<\infty$ for a.e. $r$ since $Q_T(\rho)<\infty$. By Lemma \ref{lem_B_6}, for a.e. $r$, $\nabla \n*\rho(r)\in L^4(\t)$ so that by property of smooth mollifiers,
\begin{equation}\label{equ_lem_etrest7}
\lim_{n\to\infty}\|(J_n*J_n*\nabla \n*\rho)(r)-\nabla \n*\rho(r)\|_4=0.
\end{equation}
By Holder's inequality, for a.e. $r$,
\begin{equation}\label{equ_lem_etrest6}
\begin{aligned}
&\left|\int_{\t}(J_n*J_n*\nabla \n*\rho)(r,x)\cdot u(r,x) \rho(r,x)dx\right|\\
&=\left|\int_{\t}(J_n*J_n*\nabla \n*\rho-\nabla \n*\rho)(r,x)\cdot u(r,x) \rho(r,x)dx\right|\\
&\leq  \|\rho(r)\|_2\|u(r)\|_4\|(J_n*J_n*\nabla \n*\rho)(r)-\nabla \n*\rho(r)\|_4\to 0
\end{aligned}
\end{equation}
where we used the fact that $(\nabla \n*\rho)(r,x)\cdot u(r,x)=0.$
So far we have proved that
$$
\lim_{n\to\infty}\int_{\t}(J_n*J_n*\nabla \n*\rho)(r,x)\cdot u(r,x) \rho(r,x)dx=0\text{ for a.e. }r.
$$
Apply Lemma \ref{lem_B_6} along with Jensen's inequality to (\ref{equ_lem_etrest6}),
$$
\left|\int_{\t}(J_n*J_n*\nabla \n*\rho)(r,x)\cdot u(r,x) \rho(r,x)dx\right|\leq
C\|\rho(r)\|_2^2.
$$
Since $Q_T(\rho)<\infty$, by dominated convergence theorem, (\ref{equ_lem_etrest3}) follows.

Turing to (\ref{equ_lem_etrest4}), by Holder's inequality,
\begin{equation}\label{equ_lem_etrest5}
\begin{aligned}
&\bigg|\int_{\t}(J_n*J_n*\nabla \n*\rho-\nabla\n*\rho)(r,x)\cdot \hat{\nabla}p(r,x) \rho(r,x)dx\bigg|\\
&\leq  \left(\int_{\t}|\hat{\nabla}p(r,x)|^2 \rho(r,x)dx\right)^{\frac{1}{2}}\|\rho(r)\|_2^{\frac{1}{2}}
\|(J_n*J_n*\nabla \n*\rho-\nabla\n*\rho)(r)\|_4.
\end{aligned}
\end{equation}
The condition $\overline{\act}_T(\rho)<\infty$ along with (\ref{equ_ovlact}) implies $\int_{\t}|\hat{\nabla}p(r,x)|^2 \rho(r,x)dx<\infty$ for a.e. $r$. Thus by (\ref{equ_lem_etrest7}) the point-wise convergence of  $\int_{\t}(J_n*J_n*\nabla \n*\rho)(r,x)\cdot \hat{\nabla}p(r,x) \rho(r,x)dx$ holds. Finally, by Lemma \ref{lem_B_6} and Jensen's inequality, the RHS of (\ref{equ_lem_etrest5}) is bounded by $\int_{\t}|\hat{\nabla}p(r,x)|^2 \rho(r,x)dx+C\|\rho(r)\|_2^2,$ which is integrable. By dominated convergence theorem, we reach (\ref{equ_lem_etrest4}).
\end{proof}
Let the entropy functional be defined by $$
S(\gamma)=\left\{
\begin{aligned}
&\int_{\t}\gamma(x)\log \gamma(x) dx,&\text{if }\gamma(dx)=\gamma(x)dx,\\
&\infty,&\text{otherwise},
\end{aligned}
\right.
$$
and the Fisher information functional be defined by $$I(\gamma)=\left\{
\begin{aligned}
&\int_{\t}\frac{|\nabla \gamma(x)|^2}{\gamma(x)}dx,&\text{if }\gamma(dx)=\gamma(x)dx \text{ and }\nabla \gamma \in L^1(\t),\\
&\infty,&otherwise.
\end{aligned}
\right.
$$
\begin{lem}\label{lem_decay}
If $\rho$ satisfies assumptions in Lemma \ref{lem_etrest}, then for each $t\in(0,T]$,
\begin{equation}\label{equ_decay}
S(\rho(t))\leq 2\overline{\act}_T(\rho)+\log\left(\frac{2Q_T(\rho)}{\nu t}+1\right).
\end{equation}
In addition, for each $\alpha>0$,
\begin{equation}\label{equ_decay1}
\int_{0}^{T}t^\alpha I(\rho(t))dt<\infty.
\end{equation}
\end{lem}
\begin{proof}
We start with proving (\ref{equ_decay}). Still using the smooth mollifiers $J_n$ in Lemma \ref{lem_etrest}, let $u=\k*\rho,\hat{\rho}_n=J_n*\rho_n$. By (\ref{equ_std_3}),
\begin{equation}\label{equ_lem_decay_1}
\begin{aligned}
&te^{S(\hat{\rho}_n(t))}=\int_0^t \partial_r\left[re^{S(\hat{\rho}_n(r))}\right]dr\\
&=\int_{0}^{t}e^{S(\hat{\rho}_n(r))}dr
+\int_{0}^{t}re^{S(\hat{\rho}_n(r))}\lr 1+\log(\hat{\rho}_n(r)),\partial_t \hat{\rho}_n(r)\rr dr
\\&=\int_{0}^{t}e^{S(\hat{\rho}_n(r))}dr-\nu\int_{0}^{t}re^{S(\hat{\rho}_n(r))}I(\hat{\rho}_n(r)) dr\\
&+\int_{0}^{t}\int_{\t}re^{S(\hat{\rho}_n(r))}\frac{(\nabla J_n*\rho)(r,x)}{(J_n*\rho)(r,x)}\cdot (J_n*(\rho \hat{\nabla}p ))(r,x)dxdr\\
&+\int_{0}^{t}\int_{\t}re^{S(\hat{\rho}_n(r))}\frac{(\nabla J_n*\rho)(r,x)}{(J_n*\rho)(r,x)}\cdot (J_n*(\rho u ))(r,x)dxdr.\\
\end{aligned}
\end{equation}
Due to mean value inequality,
\begin{equation}\label{equ_lem_decay_3}
\begin{aligned}
&-\frac{\nu^2}{4}I(\hat{\rho}_n(r))+\frac{\nu}{2}\int_{\t}\frac{(\nabla J_n*\rho)(r,x)}{(J_n*\rho)(r,x)}\cdot (J_n*(\rho\hat{\nabla}p ))(r,x)dx\\
&\leq \frac{1}{4} \int_{\t}\left|\frac{J_n*(\rho \hat{\nabla}p)(r,x)}{(J_n*\rho)(r,x)}\right|^2(J_n*\rho)(r,x)dx\leq \frac{1}{4}\int_{\t}|\hat{\nabla}p(r,x)|^2\rho(r,x)dx,
\end{aligned}
\end{equation}
where we used Lemma 8.1.10 of \cite{RN217} to obtain the last inequality. Since $\div u=0$, we have $
\int_{\t}(\nabla J_n*\rho)(r,x)\cdot u(r,x)dx=0$ and
\begin{equation}\label{equ_lem_decay_5}
\begin{aligned}
&
\frac{\nu}{2}\int_{\t}\frac{(\nabla J_n*\rho)(r,x)}{(J_n*\rho)(r,x)}\cdot (J_n*(\rho u))(r,x)dx\\
&=
\frac{\nu}{2}\int_{\t}\frac{(\nabla J_n*\rho)(r,x)}{(J_n*\rho)(r,x)}\cdot \left(\frac{(J_n*(\rho u))(r,x)}{(J_n*\rho)(r,x)}-u(r,x)\right)(J_n*\rho)(r,x)dx
\\
&\leq \frac{1}{4}\int_{\t}\left|\frac{J_n*(\rho u)(r,x)}{(J_n*\rho)(r,x)}-u(r,x)\right|^2(J_n*\rho)(r,x)dx+\frac{\nu^2}{4}I(\hat{\rho}_n(r)).
\end{aligned}
\end{equation}
By Jensen's inequality
\begin{equation}\label{equ_lem_decay_4}
S(J_n*\rho)\leq \log\|J_n*\rho\|_{2}^2=\log(\|J_n*\rho-1\|_{2}^2+1).
\end{equation}
Applying (\ref{equ_lem_decay_3}), (\ref{equ_lem_decay_5}) (\ref{equ_lem_decay_4}) to (\ref{equ_lem_decay_1}),
we have $$
\begin{aligned}
&\frac{\nu}{2}te^{S(\hat{\rho}_n(t))}\leq Q_T(\hat{\rho}_n)+\frac{\nu}{2}t\\
&+\frac{1}{4}\int_{0}^{t}re^{S(\hat{\rho}_n(r))}\int_{\t}\bigg[ \left|\frac{J_n*(\rho u)(r,x)}{(J_n*\rho)(r,x)}-u(r,x)\right|^2(J_n*\rho)(r,x)\\
&+|\hat{\nabla}p(r,x)|^2\rho(r,x) \bigg]dxdr.
\end{aligned}
$$
Let $$\varepsilon_n=\int_{0}^{T} \int_{\t}\left|\frac{J_n*(\rho u)(r,x)}{(J_n*\rho)(r,x)}-u(r,x)\right|^2(J_n*\rho)(r,x)dxdr.$$
By Gronwall inequality $$
te^{S(\hat{\rho}_n(t))}
\leq\left(t+\frac{2}{\nu}Q_T(\rho)\right)e^{2
\overline{\act}_T(\rho)+\frac{1}{2\nu}\varepsilon_n}.
$$
Note that
\begin{equation}\label{equ_lem_decay_14}
\begin{aligned}
&\int_{\t}\left|\frac{J_n*(\rho u)(r,x)}{(J_n*\rho)(r,x)}-u(r,x)\right|^2(J_n*\rho)(r,x)dx\\
&=\int_{\t}\left|\frac{J_n*(\rho u)(r,x)}{(J_n*\rho)(r,x)}\right|^2(J_n*\rho)(r,x)dx+\int_{\t}\left|u(r,x)\right|^2(J_n*\rho)(r,x)dx\\
&-2\int_{\t}J_n*(\rho u)(r,x)\cdot u(r,x)dx.
\end{aligned}
\end{equation}
From Lemma 8.1.10 of \cite{RN217} combined with Lemma \ref{lem_B_6} and Jensen's inequality \cite{RN217}, (\ref{equ_lem_decay_14}) can be controlled by $\|\rho(r)\|_2^2$. In addition, for $r$ satisfying $\rho(r)\in L^2(\t),$ by Fatou's Lemma and Lemma 8.1.10 of \cite{RN217}, the first term converge to $\int_{\t}|u(r,x)|^2\rho(r,x)dx$. By property of mollifiers, the other two terms also converge to $\int_{\t}|u(r,x)|^2\rho(r,x)dx$. Hence (\ref{equ_lem_decay_14}) converge to $0$ for a.e $r$. Taking $n\to \infty$, by dominated convergence theorem we have $\varepsilon_n\to0,$
and by Fatou's lemma we arrive at
$$
S(\rho(t))\leq 2\overline{\act}_T(\rho)+\log\left(1+\frac{2Q_T(\rho)}{\nu t}\right).
$$

Turning to (\ref{equ_decay1}), by (\ref{equ_std_3}), for $\alpha>0,$
$$
\begin{aligned}
&t^\alpha S(\hat{\rho}_n(t))=\int_{0}^{t}r^\alpha \lr 1+\log(\hat{\rho}_n(r)),\partial_t \hat{\rho}_n(r)\rr dr
+\int_{0}^{t}\alpha r^{\alpha-1} S(\hat{\rho}_n(r))dr\\
&=\int_{0}^{t}\alpha r^{\alpha-1} S(\hat{\rho}_n(r))-\nu\int_{0}^{t}r^\alpha I(\hat{\rho}_n(r))dr\\
&+\int_{0}^{t}\int_{\t}r^\alpha\frac{(\nabla J_n*\rho)(r,x)}{(J_n*\rho)(r,x)}\cdot (J_n*(\rho \hat{\nabla}p ))(r,x)dxdr\\
&+\int_{0}^{t}\int_{\t}r^\alpha\frac{(\nabla J_n*\rho)(r,x)}{(J_n*\rho)(r,x)}\cdot (J_n*(\rho u ))(r,x)dxdr.
\end{aligned}
$$
Similarly with (\ref{equ_lem_decay_3}), (\ref{equ_lem_decay_5}), we have $$
\begin{aligned}
&-\frac{\nu}{4}I(\hat{\rho}_n(r))+\int_{\t}\frac{(\nabla J_n*\rho)(r,x)}{(J_n*\rho)(r,x)}\cdot (J_n*(\rho\hat{\nabla}p ))(r,x)dx\\
&\leq \frac{1}{\nu}\int_{\t}|\hat{\nabla}p(r,x)|^2\rho(r,x)dx;\\
&-\frac{\nu}{4}I(\hat{\rho}_n(r))
+\int_{\t}\frac{(\nabla J_n*\rho)(r,x)}{(J_n*\rho)(r,x)}\cdot (J_n*(\rho u))(r,x)dx\\
&\leq \frac{1}{\nu}\int_{\t}\left|\frac{J_n*(\rho u)(r,x)}{(J_n*\rho)(r,x)}-u(r,x)\right|^2(J_n*\rho)(r,x)dx=\frac{\varepsilon_n}{\nu}.
\end{aligned}
$$
Hence,
$$
0\leq t^\alpha S(\hat{\rho}_n(t))\leq \int_{0}^{t}\alpha r^{\alpha-1} S(\hat{\rho}_n(r))dr-\frac{\nu}{2}\int_{0}^{t}r^\alpha I(\hat{\rho}_n(r))dr+t^\alpha\left(4\overline{\act}_T(\rho)+\frac{\varepsilon_n}{\nu}\right).
$$
Then by Jensen's inequality and Fatou's Lemma,
$$
0\leq \int_{0}^{t}\alpha r^{\alpha-1} S(\rho(r))dr-\frac{\nu}{2}\int_{0}^{t}r^\alpha I(\rho(r))dr+4t^\alpha\overline{\act}_T(\rho).
$$
The result is now an immediate consequence of  (\ref{equ_decay}).
\end{proof}
\subsection{Proof of Lemma \ref{lem_vrotrf}}
\begin{proof}
Suppose $\rho\in C([0,T];\p(\t))$, $Q_T(\rho)<\infty$ and $\overline{\act}_T(\rho)<\infty$. Take $\hat{\nabla}p$ in Lemma \ref{lem_etrest}. Let $u=\k*\rho$ and then in distribution sense,
$$
\partial_t \rho +\div (\rho u)-\nu \Delta \rho+\div (\rho\hat{\nabla}p)=0.
$$
By Lemma 8.3.1 of \cite{RN217}, to show $\rho\in AC((0,T);\p(\t))$, we just need to prove
\begin{equation}\label{equ_defpt}
\int_{0}^T\|\nu \Delta \rho(t)-\div (\rho(t)\hat{\nabla}p(t))-\div (\rho(t) u(t))\|_{-1,\rho(t)}dt<\infty.
\end{equation}
By Lemma \ref{lem_hm1r1}, Proposition \ref{prop_hm1r4} and Lemma \ref{lem_B_6},
$$
\begin{aligned}
&\int_{0}^T\|\nu \Delta \rho(t)-\div (\rho(t)\hat{\nabla}p(t))-\div (\rho(t) u(t))\|_{-1,\rho(t)}dt\\
&\leq
\int_{0}^T\bigg(\|\rho(t)\|_2+I^{\frac{1}{2}}(\rho(t))+\left(\int_{\t}|\hat{\nabla}p(t)|^2d\rho(t)\right)^{\frac{1}{2}}\bigg)dt.
\end{aligned}
$$
Then (\ref{equ_defpt}) holds by Holders inequality and Lemma \ref{lem_decay}.\\
To prove the second conclusion, note that if $\rho\in AC((0,T);\p(\t))$, by Lemma \ref{lem_hm1r1},
$$
\begin{aligned}
&\overline{\act}_T(\rho)= \frac{1}{4\nu}\int_{0}^{T}\int_{\t}|\hat{\nabla}p(t)|^2d\rho(t)dt\\
&\geq
\frac{1}{4\nu}\int_0^T\|\partial_t \rho(t)-\nu\Delta \rho(t)+\div (\rho(t) u(t))\|^2_{-1,\rho(t)}dt=\act_T(\rho).
\end{aligned}
$$
By Lemma D.34 of \cite{RN90}, for almost every $t\in[0,T]$, there exists $\tilde{\nabla}p(t)$ such that
$$
\lr \partial_t \rho(t)-\nu\Delta \rho(t)+\div (\rho(t) u(t)),\varphi\rr=\int_{\t}\nabla \varphi\cdot \tilde{\nabla}p(t)d\rho(t),\quad \forall \varphi\in C^\infty(\t),
$$
and $$
\|-\div( \rho(t)\tilde{\nabla}p(t))\|^2_{-1,\rho(t)}=\int_{\t}|\tilde{\nabla}p(t)|^2d\rho(t).
$$
By definition of $\overline{\act}_T(\rho)$ and Cauchy-Schwarz inequality,
$$
\begin{aligned}
&\overline{\act}_T(\rho)\leq  \sup_{\phi}\bigg[
\bigg(\int_0^T\int_{\t}|\tilde{\nabla}p(t)|^2d\rho(t)dt\bigg)^{\frac{1}{2}}\bigg(\int_0^T\int_{\t}|\nabla \phi(t)|^2d\rho(t)dt\bigg)^{\frac{1}{2}}\\
&-\nu\int_{\t}|\nabla \phi(t)|^2d\rho(t)dt\bigg]\leq \frac{1}{4\nu}\int_0^T\int_{\t}|\tilde{\nabla}p(t)|^2d\rho(t)dt\leq
\act_T(\rho).
\end{aligned}
$$
\end{proof}

\section{Perturbed dynamics and nice trajectory approximation}\label{sec_lln}
In this section, we establish the law of large number (LLN) for the perturbed systems (\ref{equ_sde1}) and prove the "nice" trajectory approximation. Section \ref{sec_uniq} investigates the uniqueness of perturbed mean-field equation (\ref{def_weaks}). Section \ref{sec_peest1} provides a prior energy estimation similar to Lemma \ref{lem_Q1plus} which is crucial for proving LLN. Section \ref{sec_mfd} presents the proof of Lemmas \ref{lem_rucv}, \ref{lem_meanfield} and \ref{lem_lmfd}. Finally, we reach Lemma \ref{lem_density} in Section \ref{sec_density}.
\subsection{Uniqueness of perturbed mean-field equation}\label{sec_uniq}
We prove the uniqueness of weak solution of (\ref{def_weaks}), the mean-field equation for the stochastic interacting models perturbed by $v$.
\begin{lem}\label{lem_uniq}
Given $\gamma\in\p(\t)$ with $e(\gamma)<\infty$. Suppose $\rho$ and $\rho'$ are two weak solutions of (\ref{def_weaks}) for $v\in L^\infty([0,T]\times\t;\r^2)$ and $\rho(0)=\rho'(0)=\gamma$. Then $\rho'=\rho.$
\end{lem}
\begin{proof}
Write $u=\k*\rho,$ $u'=\k*\rho'.$ Take smooth mollifiers $J_n$ defined in Lemma \ref{lem_etrest}. Then by (\ref{equ_prop_n2}) and with the same argument as the steps 2\&3 in the proof of Lemma \ref{lem_etrest}, we obtain
\begin{equation}\label{equ_swh0}
\begin{aligned}
&\frac{1}{2}\|J_n*(u(t)-u'(t))\|_2^2=-\nu\int_{0}^{t}\|J_n*(\rho(s)-\rho'(s))\|^2_{2}ds\\
&-\int_{0}^{t}\lr (\rho(s)-\rho'(s)) v(s),J_n*J_n*\n*(\rho(s)-\rho'(s))\rr ds\\
&-\int_{0}^{t}\lr\rho(s) u(s)-\rho'(s) u'(s),J_n*J_n*\nabla \n*(\rho(s)-\rho'(s))\rr ds
\end{aligned}
\end{equation}
By Lemma \ref{lem_B_6} and Jensen's inequality, $$
\begin{aligned}
&\big|\lr\rho(t) u(t)-\rho'(t) u'(t),J_n*J_n*\nabla \n*(\rho(t)-\rho'(t))\rr \big|\\
&\leq 4\sup_{(\tilde{\rho},\tilde{u})\in \{\rho(t),\rho'(t)\}\times\{u(t),u'(t)\}}\int_{\t} \tilde{u}^2d\tilde{\rho}
\leq C\left(\|\rho(t)-1\|_{2}^2+\|\rho'(t)-1\|_{2}^2\right).
\end{aligned}
$$
Similarly,
$$
\begin{aligned}
&\big|\lr (\rho(t)-\rho'(t)) v(t),J_n*J_n*\n*(\rho(t)-\rho'(t))\rr\big|
\\&\leq C_{v}\left(\|\rho(t)\|_{2}+\|\rho'(t)\|_{2}\right)\left(\|u(t)\|_2+\|u'(t)\|_2)\right).
\end{aligned}
$$
Since $Q_T(\rho)$ and $Q_T(\rho')$ are finite, applying dominated convergence theorem to (\ref{equ_swh0}), we have
\begin{equation}\label{equ_swh1}
\begin{aligned}
&\frac{1}{2}\|u(t)-u'(t)\|_2^2=-\nu\int_{0}^{t}\|\rho(s)-\rho'(s)\|^2_{2}ds\\
&-\int_{0}^{t}\lr (\rho(s)-\rho'(s)) v(s),\nabla \n*(\rho(s)-\rho'(s))\rr ds\\
&-\int_{0}^{t}\lr\rho(s) u(s)-\rho'(s) u'(s),\nabla \n*(\rho(s)-\rho'(s))\rr ds
\end{aligned}
\end{equation}
Noticing that $\div (\rho(t) u(t))=\curl(u(t)\cdot \nabla u(t)),$ we have $$
\begin{aligned}
&\bigg|\int_{0}^{t}\lr\rho(s) u(s)-\rho'(s) u'(s),\nabla \n*(\rho-\rho')(s)\rr ds \bigg|\\
&=\bigg|\int_{0}^{t}\lr (u\cdot \nabla u)(s)-(u'\cdot \nabla u')(s),u(s)-u'(s)\rr ds\bigg|.
\end{aligned}
$$
Since $\div u'(t)=0,$ writing $w=u(t)-u'(t)$,
$$\lr u'(t)\cdot\nabla (u-u')(t),(u-u')(t)\rr=\frac{1}{2}\lr \nabla |w|^2,u'(t)\rr=0.$$
Hence,
\begin{equation}\label{equ_uniquem1}
\begin{aligned}
&\bigg|\int_{0}^{t}\lr (u\cdot \nabla u)(s)-(u'\cdot \nabla u')(s),u(s)-u'(s)\rr ds\bigg|\\
&=\bigg|\int_{0}^{t}\lr (u-u')(s)\cdot \nabla u(s),u(s)-u'(s)\rr ds\bigg|\\
&\leq \int_{0}^{t}\|u(s)-u'(s)\|_{4}^2\|\nabla u(s)\|_{2}ds.
\end{aligned}
\end{equation}
Apply Gagliardo–Nirenberg interpolation inequality to (\ref{equ_uniquem1}),
\begin{equation}\label{equ_unique1}
\begin{aligned}
&\bigg|\int_{0}^{t}\lr\rho(s) u(s)-\rho'(s) u'(s),\nabla \n*(\rho-\rho')(s)\rr ds \bigg|\\
&\leq \int_{0}^{t}\|u(s)-u'(s)\|_{2}\|\nabla u(s)-\nabla u'(s)\|_{2}\|\nabla u(s)\|_{2}ds,\\
&\leq 4 \int_{0}^{t}\|u(s)-u'(s)\|_{2}\|\rho(s)-\rho'(s)\|_{2}\|\rho(s)-1\|_{2}ds,
\end{aligned}
\end{equation}
where we used the fact $$\|\nabla u(t)\|^2_{2}\leq 2\sum_{i,j\in{1,2}}\|\partial_{ij}(-\n*\rho)\|_{2}^2\leq 2\|\Delta (-\n*\rho)\|_{2}^2= 2\|\rho-1\|_{2}^2.$$
On the other hand, by Cauchy--Schwarz inequality and (\ref{equ_prop_n2}),
\begin{equation}\label{equ_unique2}
\begin{aligned}
& \bigg|\int_{0}^{t}\lr (\rho(s)-\rho'(s)  v(s)),\nabla\n*(\rho(s)-\rho'(s))\rr ds\bigg|\\
&\leq C_v\int_{0}^{t}\|u(s)-u'(s)\|_{2}\|\rho(s)-\rho'(s)\|_{2} ds,
\end{aligned}
\end{equation}
Combining (\ref{equ_swh1}) with (\ref{equ_unique1}) and (\ref{equ_unique2}), we have $$
\begin{aligned}
&\frac{1}{2}\|u(t)-u'(t)\|_{2}^2
\leq -\nu\int_{0}^{t}\|\rho(s)-\rho'(s)\|^2_{2}ds\\
&+C \int_{0}^{t}\|u(s)-u'(s)\|_{2}\|\rho(s)-\rho'(s)\|_{2}(\|\rho(s)-1\|_{2}+1)ds\\
&\leq \int_{0}^{t}\frac{C^2}{4\nu}\|u(s)-u'(s)\|_{2}^2(\|\rho(s)-1\|_{2}+1)^2ds,
\end{aligned}
$$
and the uniqueness follows from Gronwall's inequality.
\end{proof}

\subsection{Prior energy estimation of perturbed dynamics}\label{sec_peest1}
To prove the law of large number for perturbed dynamics, we also need a prior energy estimation, whose proof is similar to that of Lemma \ref{lem_Q1plus}.
\begin{lem}\label{lem_Q1plus1}
There exists constants $\lambda>0$, and $C_v$ only dependent on $\|v\|_{\infty}$ such that for each sequence $(\eta_n)$ satisfying $\eta_n\in \x_n$ and $e(\g_n*\eta_n)\leq R,$
\begin{equation}\label{iequ_Q1plus1}
\begin{aligned}
\limsup_{n\to\infty}&\frac{1}{n}\log P^v_{\eta_n}\bigg\{\sup_{0\leq t\leq T} \bigg(e(\g_n*\rho_n(t))
+ \frac{\nu}{2}\int_0^t \|\g_n*\rho_n(s)-1\|_{2}^2ds\bigg)
>l\bigg\}\\
&\leq -\lambda(l-R-C_vT).
\end{aligned}
\end{equation}
\end{lem}
\begin{proof}
Recall $\n_n=G_n*\n,\omega_n=\frac{\nu}{n}(G_n(0)-1).$
By Ito's formula, as calculation in Lemma \ref{lem_Q1plus}, for any $\lambda>0,$
\begin{equation}
\begin{aligned}
&\exp\bigg\{
n\lambda\bigg[e(\g_n*\rho_n(t))-e(\g_n*\rho_n(0))
+\nu\int_{0}^{t}\|\g_n*\rho_n(s)-1\|_{2}^2ds\\&-\int_{0}^{t}\lr\nabla \n_n*\rho_n(s),\op(\rho_n(s))\rr ds-\lambda\nu\int_{0}^{t} \lr|\nabla \n_n*\rho_n(s)|^2,\rho_n(s)\rr ds-\omega_nt\\
&-\int_{0}^{t} \lr \nabla \n_n*\rho_n(s)\cdot v(s),\rho_n(s)\rr ds
\bigg]
\bigg\}
\end{aligned}
\end{equation}
is a positive continuous martingale. By Lemma \ref{lem_ville}, for each $\eta_n\in \x_n$ with $e(\g_n*\eta_n)\leq R,$
$$
\begin{aligned}
&P^v_{\eta_n}\bigg\{\sup_{0<t\leq T}\bigg[e(\g_n*\rho_n(t))-\int_{0}^{t}\lr\nabla \n_n*\rho_n(s),\op(\rho_n(s))\rr ds\\
&+\nu\int_{0}^{t}\|\g_n*\rho_n(s)-1\|_{2}^2ds-\nu \lambda\int_{0}^{t}\int_{\t}|\nabla \n_n*\rho_n(s)|^2\rho_n(s,dx)  ds\\
&-\int_{0}^{t} \lr \nabla \n_n*\rho_n(s)\cdot v(s),\rho_n(s)\rr ds\bigg]>l\bigg\}\leq e^{-n\lambda(l-R-\omega_nT)}.
\end{aligned}
$$
By Corollary \ref{cor_B_6}, there exists $C_v>0$ dependent on $\nu$ and $\|v\|_{\infty}$, such that
$$
\lr \nabla \n_n*\rho_n(t)\cdot v(t),\rho_n(t)\rr \leq C_v+\frac{\nu}{6}\|\g_n*\rho_n(t)-1\|_2^2.
$$
Combining this inequality with (\ref{equ_b6}) and taking  $\lambda$ such that $C_2 \lambda=\frac{1}{6}$, we have
$$
\begin{aligned}
&P^v_{\eta_n}\bigg\{\sup_{0<t\leq T}\bigg[e(\g_n*\rho_n(t))-\int_{0}^{t}\lr\nabla \n_n*\rho_n(s),\op(\rho_n(s))\rr ds\\
&+\frac{2\nu}{3}\int_{0}^{t}\|\g_n*\rho_n(s)-1\|_{2}^2 ds\bigg]>l\bigg\}\leq e^{-n\lambda(l-R-\omega_nT-C_vT)}.
\end{aligned}
$$
The rest is similar to the proof of Lemma \ref{lem_Q1plus}.
\end{proof}
Then we verify that there exists a sequence of $(\gamma_n)_{n\geq 1}$ satisfying the conditions of Lemma \ref{lem_Q1plus1}.
\begin{lem}\label{lem_mfd1}
For $\gamma\in\p(\t)$ with $e(\gamma)<\infty$, there exists $\gamma_n\in \x_n$ such that $$
\lim_{n\to\infty}e(\g_n*\gamma_n)=e(\gamma),\quad \lim_{n\to\infty}d(\gamma_n,\gamma)=0.
$$
\end{lem}
\begin{proof}
Take a sequence of i.i.d. random variables $(Y_i)_{i\geq 1}$ with the law $\gamma$.\\
Let $\eta_n:=\frac{1}{n}\sum_{i=1}^n\delta_{Y_i}$. Write $\n_n=G_n*\n.$ By (\ref{equ_prop_n2}) we have
\begin{equation}
\begin{aligned}
&\e \left[\|\g_n*\k*\eta_n-\k*\gamma\|_{2}^2\right]=\e \left[\frac{1}{n^2}\sum_{1\leq i,j\leq n}\n_n(Y_i-Y_j)\right]+\|\k*\gamma\|_2^2\\
&-2 \e \left[\lr \g_n*\n*\eta_n,\gamma\rr\right]\\
&=\frac{1}{n^2}\sum_{ i\neq j}\e \n_n(Y_i-Y_j)+\frac{1}{n}\n_n(0)+\|\k*\gamma\|_2^2
\\
&-\sum_{i=1}^n\frac{2}{n}\int_{\t} \e (\g_n*\n)(x-Y_i)\gamma(dx)\\
&=\frac{n-1}{n}\|\k*\g_n*\gamma\|_2^2+\|\k*\gamma\|_2^2+\frac{1}{n}\n_n(0)\\
&-2\lr(\g_n*\n)*\gamma,\gamma\rr.
\end{aligned}
\end{equation}
Note that $\lim_{n\to\infty}\left|\frac{1}{n}\n_n(0)\right|=\lim_{n\to\infty} \frac{m_n^2\|G\|_{\infty}\|\n\|_{1}}{n}= 0.$ By Jensen's inequality and Fatou's Lemma, along with (\ref{equ_prop_n3}), we have
\begin{equation}\label{equ_58}
\limsup_{n\to\infty}\e \left[\|\k*(\g_n*\eta_n-\gamma)\|_{2}^2\right]\leq 2\|\k*\gamma\|_2^2-2\lr\n*\gamma,\gamma\rr=0.
\end{equation}
Pick $\gamma_n\in \x_n$, such that
$0\leq \|\g_n*\k*\gamma_n-\k*\gamma\|_{2}^2\leq \e \left[\|\g_n*\k*\eta_n-\k*\gamma\|_{2}^2\right]$.
In view of (\ref{equ_58}) and (\ref{equ_prop_n3}), $\lim_{n\to\infty}e(\gamma_n)=e(\gamma),$ and $\nabla \n*\gamma_n\to \nabla \n*\gamma$ in $L^2(\t)$. Hence for each $\varphi\in C^\infty(\t)$, $$
\lr \gamma_n-\gamma,\phi\rr=\lr \gamma_n-\gamma,-\Delta \n*\phi\rr=\lr \nabla \n*\gamma_n- \nabla \n*\gamma,\nabla \phi\rr\to 0,
$$
implying $\lim_{n\to\infty}d(\gamma_n,\gamma)=0.$
\end{proof}
\subsection{Proof of Lemmas \ref{lem_rucv}, \ref{lem_meanfield} and \ref{lem_lmfd}}\label{sec_mfd}

Before proving the law of large number, we give the proof for the fact that $\op$ is continuous with finite energy, which is Lemma \ref{lem_rucv}.
\begin{proof}[Proof of Lemma \ref{lem_rucv}]
By Jensen's inequality, we only need to consider the case that $\sup_{n\geq 1}e(\g_n*\gamma)<\infty.$

First we consider the case that $\varphi_n=\varphi.$ Recall $w(x,y)=r(x,y)\k(x-y)$ and write
$$f(x,y):=\left\{
\begin{aligned}
&\frac{1}{2}\frac{\varphi(x)-\varphi(y)}{r(x,y)}w(x-y),&\text{if }x\neq y;\\
&0,&\text{if }x=y.
\end{aligned}
\right.
$$
In view of (\ref{equ_rucv1}), we just need to show
$$
\int_{(\t)^2}f(x,y) \gamma_n(dx)\gamma_n(dy)\to \int_{(\t)^2}f(x,y) \gamma(dx)\gamma(dy).
$$
By Tietze extension theorem, we can construct a continuous function $f_\delta$ such that
$f_\delta(x,y)=f(x,y),$ for $r(x,y)\geq\delta$ and $\sup |f_\delta|\leq \sup |f|.$\\
By Lemma \ref{lem_pest_fen} for $\delta<\frac{1}{2}$, there exists $n_\delta$ such that for $n\geq n_\delta$,
$$
\begin{aligned}
&\bigg|\int_{(\t)^2}(f-f_\delta)(x,y) \gamma_n(dx)\gamma_n(dy)\bigg|
\leq 2\sup |f| \int_{r(x,y)<\delta}\gamma_n(dx)\gamma_n(dy)\\
&\leq 4\sup |f| \frac{4\pi e(\g_n*\gamma_n)+C_{\n}}{-\log \delta}.
\end{aligned}
$$
Similarly,
$$
\begin{aligned}
&\bigg|\int_{(\t)^2}(f-f_\delta)(x,y) \gamma(dx)\gamma(dy)\bigg|
\leq 2\sup |f| \int_{r(x,y)<\delta}\gamma(dx)\gamma(dy)\\
&\leq 2\sup |f| \frac{4\pi e(\gamma)+C_{\n}}{-\log \delta}.
\end{aligned}
$$
By lower semi-continuity of $e$ in $\p(\t)$, $e(\gamma)\leq \sup_n e(\g_n*\gamma_n)<\infty$. Hence we can take $\delta$ small enough such that
\begin{equation}\label{equ101}
\begin{aligned}
&\limsup_{n\to\infty}\bigg|\int_{(\t)^2}(f-f_\delta)(x,y) \gamma_n(dx)\gamma_n(dy)\bigg|<\frac{\varepsilon}{3},\\ &\bigg|\int_{(\t)^2}(f-f_\delta)(x,y)
\gamma(dx)\gamma(dy)\bigg|<\frac{\varepsilon}{3}.
\end{aligned}
\end{equation}
Since $f_\delta(x,y)$ is a continuous function, and $\gamma_n\otimes \gamma_n\to \gamma\otimes\gamma$ in narrow topology,
\begin{equation}{\label{equ103}}
\lim_{n\to \infty}\int_{(\t)^2}f_\delta (x,y) \gamma_n(dx)\gamma_n(dy)= \int_{(\t)^2}f_\delta (x,y)\gamma(dx)\gamma(dy).
\end{equation}
Combining (\ref{equ101}) and (\ref{equ103}), we have $$
\limsup_{n\to \infty}\left|\int_{(\t)^2}f(x,y) \gamma_n(dx)\gamma_n(dy)- \int_{(\t)^2}f(x,y) \gamma(dx)\gamma(dy)\right|<\frac{2}{3}\varepsilon.
$$
Let $\varepsilon\to 0$ and we conclude the proof for the case $\varphi_n=\varphi$.

For the general case, let
$$
f_n(x,y):=\left\{
\begin{aligned}
&\frac{1}{2}\frac{\varphi_n(x)-\varphi_n(y)}{r(x,y)}w(x-y),&\text{if }x\neq y;\\
&0,&\text{if }x=y.
\end{aligned}
\right.
$$
Note that
$$
\begin{aligned}
&|\lr \varphi_n,\op(\gamma_n)\rr-
\lr \varphi,\op(\gamma_n)\rr|
\leq \sup_{r(x,y)\geq \delta}|f(x,y)-f_n(x,y)|\\
&+\sup_{r(x,y)<\delta}|f(x,y)+f_n(x,y)|\int_{r(x,y)<\delta}\gamma_n(dx)\gamma_n(dy).
\end{aligned}
$$
Take $n\to\infty$, then we have
$$
\begin{aligned}
&\limsup_{n\to\infty}|\lr \varphi_n,\op(\gamma_n)\rr-
\lr \varphi,\op(\gamma_n)\rr|\\
&\leq \frac{1}{2} \left(\|\nabla \varphi\|_{\infty}+\|\nabla \varphi_n\|_{\infty}\right)\sup_{n\geq 1}\int_{r(x,y)<\delta}\gamma_n(dx)\gamma_n(dy).
\end{aligned}
$$
By Lemma \ref{lem_pest_fen}, take $\delta\downarrow0$ and we complete the proof.
\end{proof}
\begin{proof}[Proof of Lemma \ref{lem_meanfield}]
Given $\gamma_n\in \mathcal{X}_n$ and $e(\g_n*\gamma_n)\leq R$ and $\lim_{n\to\infty}d(\gamma_n,\gamma)=0.$
By Ito's formula, for $\phi\in C^\infty(\t)$,
\begin{equation}\label{equ_varM}
\begin{aligned}
&d\lr \phi,\rho_n(t)\rr=\lr\nu\Delta\phi,\rho_n(t)\rr dt+\lr\nabla \phi,\op(\rho_n(t))\rr dt+\lr v(t)\cdot \nabla \phi,\rho_n(t)\rr dt\\
&+\frac{\sqrt{2\nu}}{n}\sum_{i=1}^n\nabla \phi(X_i)dW_i(t).
\end{aligned}
\end{equation}
Note that by (\ref{equ_opbound})
\begin{equation}\label{equ_phirho}
\e^v_{\gamma_n}[|\lr\phi,\rho_n(t+h)-\rho_n(t)\rr|]\leq C\left(\|\nabla^2\phi\|_{\infty}+\|v\|_{\infty}\|\nabla \phi\|_{\infty}\right)h.
\end{equation}
Then by \cite{RN246} and Theorems 8.6 and 8.8 in Chapter 3 of \cite{RN151}, we conclude that $\rho_n$ is tight in $C([0,T];\p(\t))$.

By Pohorov theorem, the family of probability laws
is relatively compact in the topology of weak convergence of probability measures. We select
a convergent subsequence indexed by $n_k$. By the Skorohod representation theorem, we can construct a canonical probability space (write $\hat{P}$ for the corresponding probability measures) on which random variables $\rho,\hat{\rho}_{n}$ are defined, with the property that $\hat{\rho}_k$ has the same law as $\rho_{n_k}$ for $k = 1, 2, \cdots$ such that
\begin{equation}\label{equ_mfd2}
\hat{\rho}_k\to \rho \quad a.s.~ \hat{P} \text{ in }C([0,T];\p(\t)) \text{ as }k\to \infty.
\end{equation}
By Lemma \ref{lem_Q1plus1} and Borel-Cantelli Lemma, $$\limsup_{k\to\infty}Q_T(\g_{n_k}*\hat{\rho}_k)\leq R+C_v T \quad a.s. ~\hat{P}.$$ Then by the lower semi-continuities of $Q_T$ under the topology of $C([0,T];\p(\t))$,
\begin{equation}\label{equ_exrho0}
Q_T(\rho)\leq R+C_v T,\quad a.s.~ \hat{P}.
\end{equation}
By Lemma \ref{lem_rucv},
$$
\lim_{k\to\infty}\lr\nabla \phi(t),\op(\hat{\rho}_k(t))\rr=\lr\nabla \phi(t),\op(\rho(t))\rr \quad a.s.\; \hat{P}.
$$
By (\ref{equ_varM}) and dominated convergence theorem, for each $\phi\in C^\infty([0,T]\times \t),$
\begin{equation}\label{equ_exrho2}
\begin{aligned}
&\lim_{k\to\infty}\bigg[\lr\phi(t),\hat{\rho}_k(t)\rr-\lr\phi(s),\hat{\rho}_k(s)\rr-\int_{s}^{t}\lr\nabla \phi(r),\op(\hat{\rho}_k(r))\rr dr\\
&-\int_{s}^{t}\lr\partial_t\phi(r)+\nabla \phi(r)\cdot v(r)+\nu\Delta \phi(r),\hat{\rho}_k(r)\rr dr\bigg]\\
&=\lr\phi(t),\rho(t)\rr-\lr\phi(s),\rho(s)\rr-\int_{s}^{t}\lr\partial_t\phi(r)+\nabla \phi(r)\cdot v(r)+\nu\Delta \phi(r),\rho(r)\rr dr\\
&-\int_{s}^{t}\lr\nabla \phi(r),\op(\rho(r))\rr dr\quad a.s.\; \hat{P},
\end{aligned}
\end{equation}
and
\begin{equation}\label{equ_exrho1}
\begin{aligned}
&\hat{\e}\bigg[\lr\phi(t),\rho(t)\rr-\lr\phi(s),\rho(s)\rr-\int_{s}^{t}\lr\partial_t\phi(r)+\nabla \phi(r)\cdot v(r)+\nu\Delta \phi(r),\rho(r)\rr dr\\
&-\int_{s}^{t}\lr\nabla \phi(r),\op(\rho(r))\rr dr\bigg]=0.
\end{aligned}
\end{equation}
By Fatou's Lemma and (\ref{equ_varM}),
\begin{equation}\label{equ_exrho3}
\begin{aligned}
&\text{Var}\bigg[\lr\phi(t),\rho(t)\rr-\lr\phi(s),\rho(s)\rr-\int_{s}^{t}\lr\partial_t\phi(r)+\nabla \phi(r)\cdot v(r)+\nu\Delta \phi(r),\rho(r)\rr dr\\
&-\int_{s}^{t}\lr\nabla \phi(r),\op(\rho(r))\rr dr\bigg]\\
&\leq \liminf_{k\to\infty}\text{Var}\bigg[-\int_{s}^{t}\lr\partial_t\phi(r)+\nabla \phi(r)\cdot v(r)+\nu\Delta \phi(r),\hat{\rho}_k(r)\rr dr\\
&+\lr\phi(t),\hat{\rho}_k(t)\rr-\lr\phi(s),\hat{\rho}_k(s)\rr-\int_{s}^{t}\lr\nabla \phi(r),\op(\hat{\rho}_k(r))\rr dr\bigg]\\
&\leq \lim_{k\to\infty}\frac{2\nu}{n_k}\|\nabla \phi\|_{\infty}^2=0.
\end{aligned}
\end{equation}
Since $\lim_{n\to\infty}d(\rho_n(0),\gamma)=\lim_{n\to\infty}d(\gamma_n,\gamma)=0$, we have $\rho(0)=\gamma$. Then by (\ref{equ_exrho0}), (\ref{equ_exrho1}), (\ref{equ_exrho2}) and (\ref{equ_exrho3}), $\rho$ is almost surely a weak solution of (\ref{def_weaks}) for $v$ in Definition \ref{def_2_1}.

By Lemma \ref{lem_mfd1}, the sequence $\gamma_n$ does exists for $R> e(\gamma)$. According to the uniqueness,
i.e. Lemma \ref{lem_uniq}, $\rho$ is not stochastic and we have determined an unique weak solution $\rho^{\gamma,v}$. Finally we conclude the proof by the relatively compactness of law of $\rho_n$.
\end{proof}

\begin{proof}[Proof of Lemma \ref{lem_lmfd}]

We start with proving that there exists $\delta \in (0,\varepsilon)$ such that for any $\gamma_n\in B_\delta(\gamma)\cap \mathcal{X}_n$ and $e(\g_n*\gamma_n)\leq R$,
\begin{equation}\label{equ_mlfd1}
\lim_{n\to \infty}P^v_{\gamma_n}\left(\sup_{0\leq t\leq T} d(\rho_n(t),\rho^{\gamma,v}(t))>\varepsilon\right)= 0.
\end{equation}
To prove it by contradiction, suppose that it's not true. Then, there exists $\varepsilon>0,R>0,$ so that for any $\delta$ we can find $\gamma_n\in B_\delta(\gamma)\cap \mathcal{X}_n$ and $e(\g_n*\gamma_n)\leq R$ such that
$$
\limsup_{n\to \infty}P^v_{\gamma_n}\left(\sup_{0\leq t\leq T} d(\rho_n(t),\rho^{\gamma,v}(t))>\varepsilon\right)>0.
$$
Taking subsequence if necessary, assume $\lim_{n\to\infty}d(\gamma_n,\gamma')=0$, where $\gamma'=\gamma'(\delta)$ depend on $\delta$. Then by Lemma \ref{lem_meanfield},
$$
\lim_{n\to \infty}P^v_{\gamma_n}\left(\sup_{0\leq t\leq T} d(\rho_n(t),\rho^{\gamma',v}(t))>\frac{\varepsilon}{2}\right)=0.
$$
Hence we obtain
\begin{equation}\label{equ_mfd3}
\sup_{0\leq t\leq T} d\left(\rho^{\gamma,v}(t),\rho^{\gamma',v}(t)\right)\geq \frac{\varepsilon}{2}.
\end{equation}
For $\delta=\frac{1}{n}$, we write $\eta_n=\gamma'(\frac{1}{n})\in \overline{B}_{\frac{1}{n}}(\gamma)$ satisfying (\ref{equ_mfd3}). Since we have the estimation $$Q_T(\rho^{\eta_n,v})\leq R+C_v T,
$$
with a similar argument as the proof of Lemma \ref{lem_meanfield},
$$
\rho^{\eta_n,v}\to \rho \text{ in }C([0,T];\p(\t)),
$$
and $\rho$ is a weak solution of (\ref{def_weaks}). By Lemma \ref{lem_uniq}, $\rho=\rho^{\gamma,v}$. However, it contradicts (\ref{equ_mfd3}), so we have proved (\ref{equ_mlfd1}).

Now turn to (\ref{equ_mfd1}). By Lemma 1.14 of \cite{RN187}, $$
\begin{aligned}
&\left|\e^v_{\gamma_n}\left[\int_0^T \int_{\t} |v(t)|^2 d\rho_n(t)dt\right]-\int_0^T \int_{\t} |v(t)|^2 d\rho^{\gamma,v}(t)dt\right|
\leq \varepsilon T\sup_{0\leq t\leq T} \|v^2(t)\|_{\text{Lip}}\\
&+T\|v\|_{\infty}^2P\left(\sup_{0\leq t\leq T} d(\rho_n(t),\rho^{\gamma,v}(t))>\varepsilon\right).
\end{aligned}
$$
By definition, $\frac{1}{4\nu}\int_0^T \int_{\t} |v(t)|^2 d\rho^{\gamma,v}(t)dt=\frac{1}{4\nu}\int_0^T\int_{\t} |\nabla p(t)|^2 d\rho^{\gamma,v}(t)dt=\act_T(\rho^{\gamma,v}),$
and then the conclusion follows by (\ref{equ_mlfd1}).
\end{proof}

\subsection{Proof of Lemma \ref{lem_density}}\label{sec_density}
To prove Lemma \ref{lem_density}, we need certain properties of the special case with $v=0$ in (\ref{def_weaks}).
\begin{lem}\label{lem_ns1}
For each $\varepsilon>0$, $\rho^{\gamma,0}$ is smooth on $([\varepsilon,T]\times \t)$ and $\inf_{x} \rho^{\gamma,0}(\varepsilon,x)>0.$
\end{lem}
\begin{proof}
Let $u(t):=\k*\rho^{\gamma,0}$. When $v=0$, it's a classical result that the solution of (\ref{def_weaks}) can be obtained by (\ref{equ_ns}) on $\t$. $Q_T(\rho^{\gamma,0})<\infty$ implies $u(t)\in L^2(\t)$ and $\rho^{\gamma,0}\in L^2([0,T];L^2(\t))$. Actually, $Q_T(\rho^{\gamma,0})<\infty$ implies $\rho^{\gamma,0}$ is the Leray solution \cite{RN205} for two dimensional Navier Stokes equation, which is very regular. We will use the results in \cite{RN192} and \cite{RN190,RN268}, which work on $\r^2$ as well as on $\t$, to prove this Lemma.

By Lemma \ref{lem_decay}, for each $T$, take $\varepsilon_1\in (0,\varepsilon)$ such that $S(\rho^{\gamma,0}(\varepsilon_1))<\infty$ and $$
\int_{\varepsilon_1}^{T}I(\rho^{\gamma,0}(t))dt<\infty.
$$
Similar with Lemma 3.2 of \cite{RN192}, for $q\in [1,2),$
$$
\|\nabla \rho^{\gamma,0}(t)\|_q\leq C_q I(\rho^{\gamma,0}(t))^{3/2-1/q}.
$$
Hence $$\nabla \rho^{\gamma,0}\in L^{2q/(3q-2)}((\varepsilon_1,T);L^q(\t)).$$
Then by Theorem 2.5 of \cite{RN192}, $$
\rho^{\gamma,0}\in C([\varepsilon_1,\infty);L^1(\t)\cap C((\varepsilon_1,\infty);L^\infty(\t)).
$$
This meets the assumptions of the theorem of \cite{RN268} (which improves Theorem B of \cite{RN190}), so $\rho^{\gamma,0}$ is a smooth classical solution on $(\varepsilon_1,T]\times \t$.

To show $\rho^{\gamma,0}(\varepsilon)>0$, note that for $\rho^{\gamma,0}$ can be see as a solution of second-order parabolic equation
$$
\partial_t \rho -\nu\Delta \rho+\nabla \rho\cdot u=0.
$$
Since we have proved $\rho^{\gamma,0}$ is smooth on $\left(\left[\frac{\varepsilon}{2},T\right]\times \t\right)$ and $\rho^{\gamma,0}(\frac{\varepsilon}{2})\geq 0$, we conclude the proof by strong maximum principle of parabolic equation.
\end{proof}
\begin{lem}\label{lem_ns2}
\begin{equation}\label{equ_ns1}
e(\gamma)-e(\rho^{\gamma,0}(t))=\nu\int_{0}^{t}\|\rho^{\gamma,0}(s)-1\|_2^2ds;
\end{equation}
There exists a constant $C>0$, such that for each smooth mollifier $J$,
\begin{equation}\label{equ_ns2}
\int_0^t I(J*\rho^{\gamma,0}(s))ds\leq\frac{2}{\nu} (S(J*\gamma)-S(J*\rho^{\gamma,0}(t)))+C(e(\gamma)-e(\rho^{\gamma,0}(t))).
\end{equation}
\end{lem}
\begin{proof}
(\ref{equ_ns1}) follows by Lemma \ref{lem_etrest}. Write $u(t)=\k*\rho^{\gamma,0}(t)$. By (\ref{equ_std_3}),
$$
\begin{aligned}
&S(J*\rho^{\gamma,0}(t))-S(J*\gamma)=-\nu\int_{0}^{t} I(J*\rho^{\gamma,0}(s))ds\\
&+\int_{0}^{t}\int_{\t}\frac{(\nabla J*\rho^{\gamma,0})(t,x)}{(J*\rho^{\gamma,0})(t,x)}\cdot \frac{J*(\rho^{\gamma,0} u)(t,x)}{(J*\rho^{\gamma,0})(t,x)}(J*\rho^{\gamma,0})(t,x)dxdt.
\end{aligned}
$$
Using the inequality $ab\leq \frac{\nu}{2}a^2+\frac{1}{2\nu}b^2$ along with Lemma 8.1.10 of \cite{RN217}, we have
$$
\begin{aligned}
&S(J*\rho^{\gamma,0}(t))-S(J*\gamma)\leq -\frac{\nu}{2}\int_{0}^{t} I(J*\rho^{\gamma,0}(s))ds
\\
&
+\frac{1}{2\nu}\int_{0}^{t}\int_{\t}\bigg| \frac{J*(\rho^{\gamma,0} u)(t,x)}{(J*\rho^{\gamma,0})(t,x)}\bigg|^2(J*\rho^{\gamma,0})(t,x)dxdt\\
&\leq -\frac{\nu}{2}\int_{0}^{t} I(J*\rho^{\gamma,0}(s))ds
+\frac{1}{2\nu}\int_{0}^{t}\int_{\t} |u(t,x)|^2\rho^{\gamma,0}(t,x)dxdt.
\end{aligned}
$$
Then (\ref{equ_ns2}) follows Lemma \ref{lem_B_6} and (\ref{equ_ns1}).
\end{proof}

\begin{proof}[Proof of Lemma \ref{lem_density}]
We split the proof into three steps.\\
Step 1. Construction.\\
Write $\Phi_t$ as the heat kernel on $\t$, defined in Appendix \ref{sec_htkernel}. Given $t_1,t_2>0$ and $2t_1+t_2\leq T$, define
$$
\tilde{\rho}(t):=\tilde{\rho}_{t_1,t_2}(t)=
\left\{
\begin{aligned}
&\rho^{\gamma,0}(t),&\text{ if }0\leq t<t_1,\\
&\Phi_{\nu (t-t_1)}*\rho^{\gamma,0}(t_1),&\text{ if }t_1\leq t<t_1+t_2,\\
&\Phi_{\nu t_2}*\rho^{\gamma,0}(2t_1+t_2-t),&\text{ if }t_1+t_2\leq t<2t_1+t_2\\
&\Phi_{\nu t_2}*\rho(t-2t_1-t_2),&\text{ if }2t_1+t_2\leq t\leq T.
\end{aligned}
\right.
$$
One can easily check that
$\tilde{\rho}\in C([0,T];\p(\t))$. Writing $\tilde{u}=\k*\tilde{\rho}$, we construct $v(t)=\nabla p(t)$ below corresponding to $\tilde{\rho}(t)$ such that
\begin{equation}\label{equ_ell}
\partial_t \tilde{\rho}(t)-\nu \Delta \tilde{\rho}(t)+\div (\tilde{\rho}(t)\tilde{u}(t))
+\div (\tilde{\rho}(t) \nabla p(t))=0.
\end{equation}
For $t\in [0,t_1]$, take $p(t)=0$ and (\ref{equ_ell}) holds. Since we only care the weak solution, we just define $p(t_1+t_2)=p(2t_1+t_2)=0$, which would not bring any problems.\\
By Lemma \ref{lem_ns1}, $\tilde{\rho}(t_1)$ is smooth and $\inf_{x}\tilde{\rho}(t_1,x)>0$, so for $t_1< t<t_1+t_2$, $\tilde{\rho}(t)$ is uniform elliptic. Hence (\ref{equ_ell}) can be seen as a second-order elliptic equation for $p(t)$.  We expect all the coefficients in (\ref{equ_ell}) as an equation for $p(t)$ is regular (i.e. their derivatives of arbitrary order are uniform bounded), which ensures an unique weak solution $p(t)$ such that $\|\nabla p(t)\|_{\infty}+\|\nabla^2p(t) \|_{\infty}$ is uniform bounded. However, $\partial_t \tilde{\rho}(t)$ may not be regular when $t$ is approaching $t_1$. But we notice that $\Phi_t$ is the heat kernel, which implies $$
\partial_t \tilde{\rho}(t)=\nu\Delta \tilde{\rho}(t),\quad \forall t\in (t_1,t_2).
$$
Then (\ref{equ_ell}) reduces to
\begin{equation}\label{equ_ell1}
\div (\tilde{\rho}(t)\tilde{u}(t))+\nabla \tilde{\rho}(t)\cdot \nabla p(t)+\tilde{\rho}(t)\Delta p(t)=0,\quad \forall t\in (t_1,t_2).
\end{equation}
Pointwisely for $t$, take $p(t)$ as the weak solution of (\ref{equ_ell1}). Due to smoothness of $\tilde{\rho}(t_1)$, all the coefficients in (\ref{equ_ell1}) is regular, so $\|\nabla p(t)\|_{\infty}+\|\nabla^2p(t)
\|_{\infty}$ is uniform bounded and for $t_1< t<t_1+t_2$, (\ref{equ_ell}) holds.

For $t\geq t_1+t_2,$ because of the convolution by $\Phi_{\nu t_2}$, $\tilde{\rho}(t)$ is uniform elliptic. Regard (\ref{equ_ell}) as a second-order elliptic equations for $p(t)$ and take $p(t)$ as its weak solution. For $t\in (t_1+t_2,2t_1+t_2)$,
\begin{equation}\label{equ_ell3}
\begin{aligned}
&\partial_t \tilde{\rho}(t)=-\Phi_{\nu t_2}*\partial_t\rho^{\gamma,0}(2t_1+t_2-t)\\
&=-\nu\Phi_{\nu t_2}*\Delta \rho^{\gamma,0}(2t_1+t_2-t)+\Phi_{\nu t_2}*\div(\op(\rho^{\gamma,0}(2t_1+t_2-t)))\\
&=-\nu\Delta \tilde{\rho}(t)+\Phi_{\nu t_2}*\div(\op(\rho^{\gamma,0}(2t_1+t_2-t))),
\end{aligned}
\end{equation}
which is regular, and by (\ref{equ_std_3}), $\partial_t \tilde{\rho}(t)$ is also regular in $(2t_1+t_2,T)$, so  $\|\nabla p(t)\|_{\infty}+\|\nabla^2p(t)
\|_{\infty}$ is uniform bounded when $t\geq t_1+t_2$.

Therefore (\ref{equ_ell}) holds for the constructed pairs $(\tilde{\rho},p)$ and $\|\nabla p(t)\|_{\infty}+\|\nabla^2p(t)
\|_{\infty}$ is uniform bounded. As a consequence, $\nabla p\in \mathcal{J}$ and $\tilde{\rho}$ is a weak solution of (\ref{def_weaks}) for $\nabla p.$

Step 2. Verify that the limit of $\act_T(\tilde{\rho})$, when taking $t_1\downarrow 0$ first and then letting $t_2\downarrow 0$, is bounded above by $\act_T(\rho).$\\
Without loss of generality, we assume $2t_1+t_2<T.$
Write $\tilde{u}(t)=\k*\tilde{\rho}(t),u(t)=\k*\rho(t)$. By definition,
\begin{equation}\label{equ_final1}
\begin{aligned}
&4\nu\act_T(\tilde{\rho})=
\left(\int_{t_1}^{t_1+t_2}
+\int_{t_1+t_2}^{2t_1+t_2}
+\int_{2t_1+t_2}^{T}\right)
\|\partial_t \tilde{\rho}(t)-\nu \Delta \tilde{\rho}(t)+\div (\tilde{\rho}(t)\tilde{u}(t))
\|^2_{-1,\tilde{\rho}(t)}dt\\
&=\uppercase\expandafter{\romannumeral1}+\uppercase\expandafter{\romannumeral2}+\uppercase\expandafter{\romannumeral3}.
\end{aligned}
\end{equation}
For $t\in (t_1,t_1+t_2)$, $\partial_t \tilde{\rho}(t)=\nu \Delta \tilde{\rho}(t),$ we have
$$\uppercase\expandafter{\romannumeral1}=\int_{t_1}^{t_1+t_2}\|\div (\tilde{\rho}(t)\tilde{u}(t)\|^2_{-1,\tilde{\rho}(t)}dt.$$
By Proportion \ref{prop_hm1r4} and Lemma \ref{lem_B_6},
$$
\uppercase\expandafter{\romannumeral1}\leq C\int_{t_1}^{t_1+t_2}\|\tilde{\rho}(t)\|_2^2dt.
$$
Since the result of Lemma \ref{lem_ns2} also holds for the solution of heat equation, $$\uppercase\expandafter{\romannumeral1}\leq C(e(\rho^{\gamma,0}(t_1))-e(\Phi_{\nu t_2}*\rho^{\gamma,0}(t_1))).$$
Noting $\rho^{\gamma,0}$ is continuous under weak topology, by Fatou's Lemma and continuity of $e(\rho^{\gamma,0})$ (as a consequence of Lemma \ref{lem_ns2}),
$$
\limsup_{t_1\to0+}\uppercase\expandafter{\romannumeral1}\leq C(e(\gamma)-e(\Phi_{\nu t_2}*\gamma)).
$$
By (\ref{equ_ell3}),
$$
\begin{aligned}
&\uppercase\expandafter{\romannumeral2}=\int_{t_1+t_2}^{2t_1+t_2}\|
-2\nu \Delta \tilde{\rho}(t)+\div (\tilde{\rho}(t)\tilde{u}(t))+\Phi_{\nu t_2}*\div(\op(\rho^{\gamma,0}(2t_1+t_2-t)))
\|_{-1,\tilde{\rho}(t)}^2dt\\
&\leq 12\int_{t_1+t_2}^{2t_1+t_2}\|\nu \Delta \tilde{\rho}(t)\|_{-1,\tilde{\rho}(t)}^2dt+3\int_{t_1+t_2}^{2t_1+t_2}\|\div(\tilde{\rho}(t)\tilde{u}(t))\|^2_{-1,\tilde{\rho}(t)}dt\\
&+3\int_{t_1+t_2}^{2t_1+t_2}\|\Phi_{\nu t_2}*\div\left(\op(\rho^{\gamma,0}(2t_1+t_2-t))\right)\|^2_{-1,\tilde{\rho}(t)}dt\\
&\leq 12\int_{t_1+t_2}^{2t_1+t_2}\|\nu \Delta \tilde{\rho}(t)\|_{-1,\tilde{\rho}(t)}^2dt+3\int_{t_1+t_2}^{2t_1+t_2}\|\div(\tilde{\rho}(t)\tilde{u}(t))\|^2_{-1,\tilde{\rho}(t)}dt\\
&+3\int_{t_1+t_2}^{2t_1+t_2}\|\div\op(\rho^{\gamma,0}(2t_1+t_2-t))\|^2_{-1,\rho^{\gamma,0}(2t_1+t_2-t)}dt,
\end{aligned}
$$
where we used Lemma \ref{lem_hm1r5} to get the last inequality. Then by Proposition \ref{prop_hm1r4} and Lemma \ref{lem_B_6} along with Jensen's inequality,
$$
\begin{aligned}
&\uppercase\expandafter{\romannumeral2}
= 12\int_{t_1+t_2}^{2t_1+t_2}I(\tilde{\rho}(t))dt
+3\int_{t_1+t_2}^{2t_1+t_2}\int_{\t}|\tilde{u}(t)|^2d\tilde{\rho}(t)dt\\
&+3\int_{t_1+t_2}^{2t_1+t_2}\int_{\t}|\k*\rho^{\gamma,0}(2t_1+t_2-t)|^2d\rho^{\gamma,0}(2t_1+t_2-t)dt
\\
&\leq  12\nu^2\int_{0}^{t_1}I( \rho^{\gamma,0}(t))dt+
C\int_{0}^{t_1}\|\rho^{\gamma,0}(t)-1\|_2^2dt,
\end{aligned}
$$
which is, in view of Lemma \ref{lem_ns2}, bounded by
$$
24\nu(S(\Phi_{\nu t_2}*\gamma)-S(\Phi_{\nu t_2}*\rho^{\gamma,0}(t_1)))+C(e(\gamma)-e(\rho^{\gamma,0}(t_1))).
$$
Also by Fatou's Lemma,
$$
\limsup_{t_1\to\infty}\uppercase\expandafter{\romannumeral2}\leq 0.
$$
By Lemma \ref{lem_hm1r5}, for each $\delta>0$
$$
\begin{aligned}
&\uppercase\expandafter{\romannumeral3}
\leq (1-\delta)^{-1}\int_{0}^{T-2t_1-t_2}\|\partial_t \rho(t)-\nu\Delta \rho(t)+\div(\rho(t) u(t))\|^2_{-1,\rho(t)}dt\\
&+\delta^{-1} \int_{0}^{T-2t_1-t_2}\int_{\t}|\Phi_{\nu t_2}*u(t)-u(t)|^2d\rho(t)dt.
\end{aligned}
$$
Combining these estimations, we have
$$
\begin{aligned}
&\limsup_{t_1\to 0+}\act_T(\tilde{\rho})\leq (1-\delta)^{-1}\act_T(\rho)
+(4\nu\delta)^{-1} \int_{0}^{T}\int_{\t}|\Phi_{\nu t_2}*u(t)-u(t)|^2d\rho(t)dt\\
&+C(e(\gamma)-e(\Phi_{\nu t_2}*\gamma)).
\end{aligned}
$$
Note that $Q_T(\rho)<\infty$ implies $\int_0^T \|\rho(t)\|_2^2dt<\infty,$ so by Lemma \ref{lem_B_6} and dominated convergence theorem, the second term vanishes when $t_2\downarrow 0$. The limit of third term is also non-positive by Fatou's Lemma. Thus
$$
\limsup_{t_2\to 0+} \limsup_{t_1\to 0+}\act_T(\tilde{\rho})\leq (1-\delta)^{-1}\act_T(\rho).$$
Due to the arbitrariness of $\delta$ we conclude that
$$
\lim_{t_2\to 0+} \lim_{t_1\to 0+}\act_T(\tilde{\rho})\leq \act_T(\rho).$$

Step 3. Convergence.

It's easy to check for each $t\in[0,T]$, $\lim_{t_2\to 0+} \lim_{t_1\to 0+}d(\tilde{\rho}(t),\rho(t))=0.$ To conclude the proof, we need strength this point-wise convergence into $$\lim_{t_2\to 0+} \lim_{t_1\to 0+}\sup_{0\leq t\leq T}d(\tilde{\rho}(t),\rho(t))=0,$$ which can be obtained by the compactness result below.
\end{proof}
\begin{lem}\label{lem_rcomp}
If $(\rho_n)_{n\geq 1}$ is a sequence of weak solutions of (\ref{def_weaks}) for $v_n=\nabla p_n\in \mathcal{J}$ with $$
\sup_{n\geq 1}\int_{0}^{T}\int_{\t}v_n^2(t,x)\rho(t,x)dxdt<\infty,
$$
then $(\rho_n)_{n\geq 1}$ is relatively compact set in $C([0,T];\p(\t)).$
\end{lem}
\begin{proof}
Suppose $$\int_{0}^{T}\int_{\t}v_n^2(t,x)\rho(t,x)dxdt\leq K$$ for certain $K\geq 0.$
Note that by (\ref{equ_opbound}) and Holder's inequality
$$
\begin{aligned}
&|\lr\phi,\rho_n(t+h)-\rho_n(t)\rr|\leq C\|\nabla^2\phi\|_{\infty}h+\int_{t}^{t+h}\int_{\t}v_n(s,x)\cdot\nabla\phi(s,x)\rho_n(s,x) dxds\\
&\leq C_{\phi}h+\left(\int_{t}^{t+h}\int_{\t}|\nabla \phi(s,x)|^2\rho(s,x)dxds\right)^{\frac{1}{2}} K^{\frac{1}{2}}\\
&\leq C_{\phi}\left(h+h^{\frac{1}{2}}K^{\frac{1}{2}}\right).
\end{aligned}
$$
Then we conclude the proof by \cite{RN246} and Theorems 8.6 and 8.8 in Chapter 3 of \cite{RN151}.
\end{proof}
\begin{appendix}
\section{Weighted Sobolev space}\label{sec_wsp}
In this section, we give some property of $\|\cdot\|_{-1,\mu}$ for $\mu\in \p(\t)$. Recall that
$$\|m\|_{-1,\mu}^2=\sup_{\phi \in C^\infty(\t)}\left\{2\lr\phi,m\rr-\int_{\t}|\nabla\phi(x)|^2dm\right\}.$$
\begin{lem}\label{lem_hm1r1}
Suppose $\int_{\t} |v|^2d\mu<\infty,$ then $$
\|-\div (\mu v)\|_{-1,\mu}^2\leq \int_{\t} |v|^2d\mu.$$
\end{lem}
\begin{proof}
By definition of $\|\cdot\|_{-1,\mu}$ and the inequality $2ab-a^2\leq b^2$, $$
\|-\div (\mu v)\|^2_{-1,\mu}= \sup_{\phi\in C^\infty(
\t)}\left\{2\int_{\t} v\cdot \nabla \phi d\mu-\int_{\t}|\nabla \phi|^2d\mu\right\}
\leq \int_{\t}|v|^2d\mu.
$$
\end{proof}
Recall $$S(\mu)=\left\{
\begin{aligned}
&\int_{\t}\mu(x)\log \mu(x)dx,&\text{if }\mu(dx)=\mu(x)dx,\\
&\infty,&otherwise,
\end{aligned}
\right.$$ and
$$I(\mu)=\left\{
\begin{aligned}
&\int_{\t}\frac{|\nabla \mu(x)|^2}{\mu(x)}dx,&\text{if }\mu(dx)=\mu(x)dx \text{ and }\nabla \mu \in L^1(\t),\\
&\infty,&otherwise.
\end{aligned}
\right.
$$
\begin{prop}\label{prop_hm1r4}
For $\mu\in \p(\t),$
$$\|\Delta \mu\|^2_{-1,\mu}=I(\mu) \text{ if }S(\mu)<\infty;$$
$$\|\div[\mu (\k*\mu)]\|^2_{-1,\mu}\leq \int_{\t}|\k*\mu|^2d\mu\text{ if }\mu\in L^2(\t).$$
\end{prop}
\begin{proof}
The first equality is proved by Theorem D.45. in \cite{RN90} in $\r^2$, which can be adapted to $\t$. The second inequality follows from Lemma \ref{lem_hm1r1} and Lemma \ref{lem_B_6}.
\end{proof}
\begin{lem}\label{lem_hm1r5}
Let $J$ be an arbitrary smooth mollifier. Then for each $\delta>0,$
\begin{equation*}
\|J*m+\div (J*\mu( J*v))\|^2_{-1,J*\mu}\leq (1-\delta)^{-1}\|m+\div(\mu v)\|^2_{-1,\mu}+\delta^{-1} \int_{\t}|J*v-v|^2d\mu.
\end{equation*}
\end{lem}
\begin{proof}
By Jensen's inequality, for each $\varphi\in C^\infty(\t)$,
$$
\begin{aligned}
&2\lr J*m+\div (J*\mu( J*v)),\varphi\rr -\int_{\t}|\nabla \varphi|^2d(J*\mu)\\
&\leq 2\lr m+\div(\mu(J*v)),J*\varphi\rr -\int_{\t}|\nabla J*\varphi|^2d\mu\\
&= 2\lr m+\div(\mu v),J*\varphi\rr
+2\lr \div(\mu (J*v-v)),J*\varphi\rr-\int_{\t}|\nabla J*\varphi|^2d\mu.
\end{aligned}
$$
Using $2ab\leq \delta a^2+\frac{1}{\delta} b^2,$
$$
2\lr \div(\mu (J*v-v)),J*\varphi\rr\leq
\delta^{-1}\int_{\t}|J*v-v|^2d\mu+\delta\int_{\t}|\nabla J*\varphi|^2d\mu.
$$
Hence,
$$
\begin{aligned}
&2\lr J*m+\div (J*\mu( J*v)),\varphi\rr -\int_{\t}|\nabla \varphi|^2d(J*\mu)\\
&\leq 2\lr m+\div(\mu v),J*\varphi\rr
+\delta^{-1}\int_{\t}|J*v-v|^2d\mu-(1-\delta)\int_{\t}|\nabla J*\varphi|^2d\mu\\
&=(1-\delta)^{-1}\left(2\lr m+\div(\mu v),(1-\delta)J*\varphi\rr-\int_{\t}|(1-\delta)\nabla J*\varphi|^2d\mu\right)\\
&+\delta^{-1}\int_{\t}|J*v-v|^2d\mu\leq \delta^{-1}\int_{\t}|J*v-v|^2d\mu+(1-\delta)^{-1} \|m+\div(\mu v)\|^2_{-1,\mu}.
\end{aligned}
$$
Taking the supremum for $\varphi$ we conclude the result.
\end{proof}
\section{Heat kernel on torus}\label{sec_htkernel}
We define the heat kernel on $\t$,
\begin{equation}\label{heatk}
\Phi(t,x)=\sum_{n,m=-\infty}^{\infty} e^{-4\pi^2(n^2+m^2)t}\exp(2\pi i (nx_1+mx_2)),\quad x=(x_1,x_2),
\end{equation}
for $t>0.$
\begin{thm}\label{heatmain}
{\rm(1)} For any $t>0,$
$\Phi_t(x):=\Phi(t,x)$ is a well defined smooth mollifier.\\
{\rm(2)} $\Phi_t(x)*\Phi_s(x)=\Phi_{t+s}(x).$\\
{\rm(3)} For $\gamma\in \p(\t)$ and $\nu>0$, $\rho(t)=\Phi_{\nu t}*\gamma$ satisfies the heat equation:
\begin{equation}\label{heat_pro1}
\frac{\partial}{\partial t}\rho(t,x)=\nu \Delta \rho(t,x),\quad t>0,
\end{equation}
and $\rho\in C([0,T];\p(\t))$.
\end{thm}
\begin{proof}
(1) Based on the theory of uniformly convergent series, $\Phi_t(x,y)$ is a well defined smooth function. Since the integral of every term in (\ref{heatmain}) is zero except the case $(n,m)=(0,0)$, $\int_{\t}\Phi_t dx=1.$ By Poisson summation formula (e.g. see Theorem 3.1.17 of \cite{RN201}),
\begin{equation}\label{equ_heatm1}
\Phi(t,x)=\sum_{n,m=-\infty}^{\infty}\frac{1}{4\pi t}e^{-\frac{(x_1-n)^2+(x_2-m)^2}{4t}}.
\end{equation}
Hence $\Phi(t,x)>0.$ We conclude that $\Phi_t(x)$ is a smooth mollifier.\\
(2) It follows by direct calculation.\\
(3) Since $\int_{\t}\Phi_t(x)dx=1$, for each $t\geq 0,$ $\rho(t)\in \p(\t)$. By (\ref{equ_heatm1}), $$\lim_{t\to0+}\int_{|x|>\delta}\Phi_t(x)dx=0.$$ Hence $\rho$ is continuous under weak topology at $0$. By (2), we have $\rho\in C([0,T];\p(\t))$.\\
For $t>0$, since each item in the sum of (\ref{heatk}) for $\nu t$ satisfies (\ref{heat_pro1}), $\Phi_{\nu t}$ also satisfies (\ref{heat_pro1}), so does $\rho(t)$.
\end{proof}
\section{Property of energy functional}\label{sec_hm1}
By (\ref{equ_N}), $\n$ is bounded from below, so $$\lr \n*\gamma,\gamma\rr=\int_{\t}\n(x-y)\gamma(dx)\gamma(dy)$$ is well-defined for any probability measure $\gamma$. By Young's inequality, for any finite signed measure $\eta$ on $\t$, $\n*\eta$,$\nabla \n*\eta$ is in $L^1(\t)$.
\begin{lem}\label{lem_basdes}
Suppose $\gamma,\eta$ are probability measures. we have\\
{\rm(1)} $\n*\gamma,\k*\gamma$ has a weak derivative and
\begin{equation}\label{equ_basdes1}
\nabla(\n*\gamma)=\nabla \n*\gamma;
\end{equation}
\begin{equation}\label{equ_basdes2}
 \div(\k*\gamma)=0,\quad \curl (\k*\gamma)=\gamma.
\end{equation}
{\rm(2)} If $J$ is a smooth mollifier,
\begin{equation}\label{equ_basdes3}
\lr\n*J*J*\gamma,\eta\rr=\lr\nabla \n*J*\gamma,\nabla \n*J*\eta\rr.
\end{equation}
\end{lem}
\begin{proof}
We start with (1). (\ref{equ_basdes1}) is standard and follows from Fubini's lemma and integration
by parts against test functions $\phi\in C^\infty(\t)$ to show the equality in the distribution sense. (\ref{equ_basdes2}) can be proved by a similar argument as the proof of Lemma 7.1 of \cite{RN171}.\\
Turing to (2), by (\ref{equ_green}), $$
J*\eta=-\Delta \n*J*\eta+1.
$$
Noting that $\int_{\t}\n dx=0$, by Fubini's lemma and integration
by parts $$
\lr\n*J*J*\gamma,\eta\rr=\lr \n*J*\gamma,-\Delta \n*J*\eta\rr=\lr \nabla \n*J*\gamma,\nabla \n*J*\eta\rr.
$$
\end{proof}
\begin{prop}\label{prop_n}
If $\gamma,\eta\in \p(\t)$ and $\nabla \n*\gamma,\nabla \n*\eta\in L^2(\t)$, then
\begin{equation}\label{equ_prop_n1}
\lr\n*\gamma,\eta\rr=\lr \nabla \n*\gamma,\nabla \n*\eta \rr;
\end{equation}
\begin{equation}\label{equ_prop_n2}
\lr\n*(\gamma-\eta),\gamma-\eta\rr=\|\nabla \n*(\gamma-\eta)\|_2^2=\|\k*(\gamma-\eta)\|_2^2.
\end{equation}
In particular, If $\gamma\in \p(\t)$, then
\begin{equation}\label{equ_prop_n3}
\lr\n*\gamma,\gamma\rr=\|\nabla \n*\gamma\|_{2}^2=\|\k*\gamma\|_{2}^2,
\end{equation}
where we admit $\infty=\infty$ if $\nabla \n*\gamma \notin L^2(\t).$
\end{prop}
\begin{proof}
We first prove (\ref{equ_prop_n1}).

Take smooth mollifiers $G_n=\g_n*\g_n$ defined in Section \ref{sec_molg}. By (\ref{equ_gnn3}),
$$G_n*\n(x)\leq \n(x)+\frac{C}{m_n^2}.$$
Combining with (\ref{equ_basdes3}),
\begin{equation}\label{equ_hm1s4}
\begin{aligned}
&\lr \n*\gamma,\eta\rr\geq \lr\n*G_n*\gamma,\eta\rr-\frac{C}{m_n^2}\\
&= \lr \n *\g_n*\gamma,\g_n*\eta\rr-\frac{C}{m_n^2}=\lr\g_n*\nabla \n*\gamma,\g_n*\nabla \n*\eta\rr-\frac{C}{m_n^2}.
\end{aligned}
\end{equation}
Since $\nabla \n*\gamma,\nabla \n*\eta\in L^2(\t)$, by property of smooth mollifiers, $\g_n*\nabla \n*\gamma\to \nabla \n*\gamma,$ $\g_n*\nabla \n*\eta\to \nabla \n*\eta$ in $L^2(\t)$, so
$$\lr \n*\gamma,\eta\rr\geq \limsup_{n\to\infty}\lr\g_n*\nabla \n*\gamma,\g_n*\nabla \n*\eta\rr= \lr\nabla \n*\gamma,\nabla \n*\eta\rr.$$
By Fatou's lemma and (\ref{equ_basdes3}),
$$
\liminf_{n\to\infty}\lr \g_n*\nabla\n*\gamma,\g_n*\nabla\n*\eta \rr=\liminf_{n\to\infty}\lr \n*G_n*\gamma,\eta\rr\geq  \lr \n*\gamma,\eta\rr.
$$
Hence,
$$
\lr \nabla \n*\gamma,\nabla \n*\eta \rr=\lim_{n\to\infty}\lr \g_n*\nabla\n*\gamma,\g_n*\nabla\n*\eta \rr\geq \lr \n*\gamma,\eta\rr.
$$
So (\ref{equ_prop_n1}) holds and
(\ref{equ_prop_n2}) is a direct corollary.

Now we turn to proving (\ref{equ_prop_n3}). Take $\gamma=\eta$ in (\ref{equ_hm1s4}) and let $n\to\infty$ and then we obtain
$$
\lr\n*\gamma,\gamma\rr\geq \limsup_{n\to\infty}\|\g_n*\nabla \n*\gamma\|_2^2.
$$
By Fatou's lemma,
$$
\lr\n*\gamma,\gamma\rr\geq \|\nabla \n*\gamma\|_2^2.
$$
Since $\g_n*\nabla \n*\gamma \in L^2(\t)$, by (\ref{equ_prop_n1}), $\lr\g_n*\n*\gamma,\g_n*\gamma\rr=\|\nabla\n* \g_n*\gamma\|_2^2$. By Fatou's lemma and Jensen's inequality,
$$\lr\n*\gamma,\gamma\rr\leq \liminf_{n\to\infty}\lr\g_n*\n*\gamma,\g_n*\gamma\rr
= \liminf_{n\to\infty} \|\nabla\n* \g_n*\gamma\|_2^2\leq \|\nabla \n*\gamma\|_2^2.$$
\end{proof}
At the end of this section, we give some estimates for $\gamma\in \p
(\t)$ with $e(\gamma)<\infty.$
\begin{lem}\label{lem_pest_fen}
There exists a constant $C_{\n}$, such that for each $\gamma\in \p(\t)$ with $e(\gamma)<\infty$ and $\delta<\frac{1}{2}$,
\begin{equation}\label{equ_pest0}
(\gamma \otimes \gamma)(\{(x,y):r(x,y)\leq \delta\})\leq \frac{C_{\n}+4\pi e(\gamma)}{-\log \delta}.
\end{equation}
\begin{equation}\label{equ_pest1}
\gamma(\overline{B}_{\delta}(x))\leq \bigg(\frac{C_{\n}+4\pi e(\gamma)}{-\log(2\delta)}\bigg)^{\frac{1}{2}}.
\end{equation}
In addition, given a smooth mollifier $J$ and a non-negative sequence $m_n\to\infty$, let $J_n(x):=m_n^2J\left(m_n^{-1}x\right)$. Then for each $\delta<\frac{1}{2}$ there exists $n_\delta$ such that for each $\gamma\in \p(\t)$ and $n\geq n_\delta$,
\begin{equation}\label{equ_pest2}
(\gamma \otimes \gamma)(\{(x,y):r(x,y)\leq \delta\})\leq \frac{2C_{\n}+8\pi e(J_n*\gamma)}{-\log \delta}.
\end{equation}
\end{lem}
\begin{proof}
By (\ref{equ_N}), take $C_{\n}$ such that $C_{\n}+2\pi\n(x-y)\geq -\log r(x,y).$
$$
\begin{aligned}
&C_{\n}+4\pi e(\gamma)\geq -\int_{\t}\log r(x,y)\gamma(dx)\gamma(dy) \\
&\geq -\inf_{|x|\leq \delta}\log(|x|)(\gamma \otimes \gamma)(\{(x,y):r(x,y)\leq \delta\}).
\end{aligned}
$$
Hence (\ref{equ_pest0}) holds.
(\ref{equ_pest1}) follows from $$
\gamma(\overline{B}_{\delta}(x))=(\gamma\otimes \gamma)^{\frac{1}{2}}(\overline{B}_{\delta}(x)\times\overline{B}_{\delta}(x) )\leq (\gamma \otimes \gamma)^{\frac{1}{2}}(\{(x,y):r(x,y)\leq 2\delta\}).
$$
By the property of smooth mollifiers, there exists $n_\delta$ such that for each $n\geq n_\delta$,
$$
C_{\n}+2\pi(J_n*J_n*\n)(x-y)\geq
\left\{
\begin{aligned}
&0, &r(x,y)>\delta,\\
&-\frac{1}{2}\log \delta, &r(x,y)<\delta,
\end{aligned}
\right.
$$
and similarly we can prove (\ref{equ_pest2}).
\end{proof}

\end{appendix}
%
%

\begin{acks}[Acknowledgments]
The authors would like to thank Prof. Jin Feng from University of Kansas for his generous help.

The second author was supported by Biomedical Pioneering Innovation Center, Peking University.
\end{acks}
\begin{funding}
The authors were supported by NSFC (No. 11971037).

\end{funding}



\bibliographystyle{imsart-number} 
\bibliography{ref}       

\end{document}